\documentclass[11pt]{article}
\usepackage{amsmath,amsthm,amssymb,amsfonts}
\usepackage{latexsym}
\usepackage{graphicx,psfrag,import}
\usepackage{fullpage}
\usepackage{framed}
\usepackage{verbatim}
\usepackage{color}
\usepackage{epsfig}
\usepackage{epstopdf}
\usepackage{hyperref}
\usepackage{geometry}
\usepackage{mathtools}
\usepackage{arydshln}
\usepackage{enumerate}
\usepackage{multicol}
\usepackage{a4wide}
\usepackage{booktabs}
\usepackage{enumitem}
\usepackage{lineno}
\usepackage{parcolumns}
\usepackage{thmtools}
\usepackage{xr}
\usepackage{epstopdf}
\usepackage{mathrsfs}
\usepackage{subfig}
\usepackage{caption}
\usepackage{comment}
\usepackage{authblk}
\usepackage{setspace}
\usepackage[normalem]{ulem}
\usepackage{footmisc}
\usepackage{float}

\theoremstyle{plain}
\newtheorem{theorem}{Theorem}[section]

\newtheorem{proposition}[theorem]{Proposition}
\newtheorem{lemma}[theorem]{Lemma}

\theoremstyle{definition}
\newtheorem{definition}[theorem]{Definition}
\newtheorem{example}[theorem]{Example}

\newtheorem{remark}[theorem]{Remark}

\theoremstyle{remark}

\newcommand\RR{\mathbb{R}}

\newcommand\by{\boldsymbol{y}}

\newcommand\byi{\boldsymbol{y_i}}
\newcommand\bypj{\boldsymbol{y'_j}}
\newcommand\byj{\boldsymbol{y_j}}

\renewcommand\bf{\boldsymbol{f}}

\newcommand\bk{\boldsymbol{k}}

\newcommand\bx{\boldsymbol{x}}
\newcommand\bv{\boldsymbol{v}}

\newcommand\bY{\boldsymbol{Y}}
\newcommand\bA{\boldsymbol{A}}
\newcommand\bB{\boldsymbol{B}}

\newcommand\bI{\boldsymbol{I}}
\newcommand\bq{\boldsymbol{q}}
\newcommand\bp{\boldsymbol{p}}
\newcommand\br{\boldsymbol{r}}

\newcommand\bJ{\boldsymbol{J}}

\newcommand\hbJ{\hat{\boldsymbol{J}}}

\newcommand{\mK}{\mathcal{K}}
\newcommand{\dK}{\mathcal{K}_{\RR\text{-disg}}}
\newcommand{\mJ}{\mathcal{J}_{\RR}}
\newcommand{\eJ}{\mathcal{J}_{\textbf{0}}}
\newcommand{\mD}{\mathcal{D}_{\textbf{0}}}

\newcommand{\mS}{\mathcal{S}}
\newcommand{\mSG}{\mathcal{S}_G}

\newcommand{\hbk}{\hat{\bk}}

\newcommand{\hbq}{\hat{\bq}}

\newcommand\bd{\boldsymbol{d}}
\newcommand{\defi}{\textbf}
\DeclareMathOperator{\spn}{span}

\begin{document}

\title{A Lower Bound on the Dimension of the $\RR$-Disguised Toric Locus of a Reaction Network}

\author[1]{
         Gheorghe Craciun%
}
\author[2]{
        Abhishek Deshpande%
}
\author[3]{
        Jiaxin Jin%
}
\affil[1]{\small Department of Mathematics and Department of Biomolecular Chemistry, 
\protect \\
University of Wisconsin-Madison}
\affil[2]{Center for Computational Natural Sciences and Bioinformatics, \protect \\
 International Institute of Information Technology Hyderabad}
\affil[3]{\small Department of Mathematics, 
\protect \\
The Ohio State University}

\date{} 

\maketitle

\begin{abstract}
\noindent
Polynomial dynamical systems (i.e. dynamical systems with polynomial right hand side) are ubiquitous in applications, especially as models of reaction networks and interaction networks. The properties of general polynomial dynamical systems can be very difficult to analyze, due to nonlinearity, bifurcations, and the possibility for chaotic dynamics. On the other hand, {\em toric dynamical systems} are polynomial dynamical systems that appear naturally as models of reaction networks, and have very robust and stable properties. A {\em disguised toric dynamical system} is a polynomial dynamical system generated by a reaction network $\mathcal N$ and some choice of positive parameters, such that (even though it may not be toric with respect to $\mathcal N$) it has a toric realization with respect to some  network $\mathcal N'$. Disguised toric dynamical systems enjoy all the robust stability properties of toric dynamical systems. In this paper, we study a larger set of dynamical systems where the rate constants are allowed to take both positive and negative values. More precisely, we analyze the \emph{$\RR$-disguised toric locus} of a reaction network $\mathcal N$, i.e., the subset in the space rate constants (positive or negative) of $\mathcal N$ for which the corresponding polynomial dynamical system is disguised toric. We focus especially on finding a lower bound on the dimension of the $\RR$-disguised toric locus. 
\end{abstract}

\tableofcontents

\section{Introduction}

Polynomial dynamical systems (i.e. dynamical systems with polynomial right hand side) are ubiquitous in models of (bio)chemical reaction networks, the spread of infectious diseases, population dynamics in ecosystems, and in many other settings. Our focus in this paper is on {\em complex balanced dynamical systems} (also known as {\em toric dynamical systems}~\cite{CraciunDickensteinShiuSturmfels2009} owing to their connection with toric varieties ~\cite{dickenstein2020algebraic}). Complex balanced dynamical systems are known to exhibit remarkably robust dynamics~\cite{horn1972general}. In particular, it is known that for complex balanced dynamical systems, there exists a unique positive steady state within each affine invariant subspace. Further, there exists a strictly convex Lyapunov function, 
which implies that all positive steady states are locally asymptotically stable~\cite{horn1972general, yu2018mathematical}. They are also related to the \emph{Global Attractor Conjecture}~\cite{CraciunDickensteinShiuSturmfels2009} which states that complex balanced dynamical systems have a globally attracting steady state within each stoichiometric compatibility class. Several special cases of this conjecture have been proved~\cite{gopalkrishnan2014geometric, pantea2012persistence, craciun2013persistence, boros2020permanence}, and a proof in full generality has been proposed in~\cite{craciun2015toric}.

A concept that is often useful in the study of reaction networks is \emph{dynamical equivalence}~\cite{craciun2008identifiability, horn1972general}, which reflects the fact that it is possible for different reaction networks to generate the same set of  differential equations, for appropriate choices of parameters (i.e., reaction rate constants). Due to the rich properties of toric dynamical systems, it is important  to study mass-action systems that are {\em dynamically equivalent to toric dynamical systems}. Such systems are called \emph{disguised toric dynamical systems}~\cite{2022disguised}. Another object of interest is the locus within the set of parameters (reaction rate constants) that generate toric dynamical systems. This is called the \emph{toric locus}, and has been studied for a long time. In particular, it is known that the toric locus (up to a change of coordinates) is a \emph{toric variety}, and its codimension is the {\em deficiency} of the network~\cite{CraciunDickensteinShiuSturmfels2009, craciun2020structure, michalek2021invitation}
Combining these concepts, one can ask: what is the locus in the space of (positive) reaction rate constants that generate {\em disguised toric dynamical systems}? This set is called the \emph{disguised toric locus}. In~\cite{2022disguised}, the authors provide several examples that illustrate the disguised toric locus and propose an algorithm to calculate this locus. 
Recently, the disguised toric locus has been shown to be invariant under invertible affine transformations of the reaction network~\cite{haque2022disguised}. We study a related concept called the \emph{$\RR$-disguised toric locus}, where the reaction rate constants are allowed to take both positive and negative values.

In this paper we introduce a method for constructing continuous injective transformations that take values on the $\RR$-disguised toric locus, and allow us to establish a lower bound on the dimension of the  $\RR$-disguised toric locus. We illustrate the calculation of this lower bound on some simple examples; interestingly, for some networks the lower bound is actually equal to the dimension of the space of parameters of the network,  {\em which implies that the  $\RR$-disguised toric locus has positive measure}. This is important, because (as we will see in examples) for the same networks  the {\em toric locus} has measure zero, so the classical results on complex balanced systems can only be applied for relatively few choices of parameter values; this shows that the extension from {\em toric} systems to {\em  $\RR$-disguised toric} systems can be very useful.   


This paper is organized as follows: 
In Sections~\ref{sec:reaction_network} and \ref{sec:complex_balanced} we introduce reaction networks, complex-balanced dynamical systems, and flux systems. In Section~\ref{sec:dynamical_equivalence} we define the notions of \emph{dynamical equivalence} and \emph{flux equivalence}. In Section~\ref{sec:disguised_locus} we formally define the  $\RR$-disguised toric locus of a reaction network. In Section~\ref{sec:main_result} we give a lower bound on the dimension of the  $\RR$-disguised toric locus and we look at several examples. In Section~\ref{sec:discussion} we discuss the main results of the paper and list possible directions for future work.

\medskip

\textbf{Notation.}
We will denote by $\mathbb{R}_{\geq 0}^n$ the set of vectors in $\mathbb{R}^n$ with non-negative entries. Similarly,  $\mathbb{R}_{>0}^n$ will denote the set of vectors in $\mathbb{R}^n$ with positive entries.  Given vectors $\bx = (\bx_1, \ldots, \bx_n)^{\intercal}\in \RR^n_{>0}$ and $\by = (\by_1, \ldots, \by_n)^{\intercal} \in \RR^n$, we define:
\begin{equation} \notag
\bx^{\by} = \bx_1^{y_{1}} \ldots \bx_n^{y_{n}}.
\end{equation}
Further, for any two vectors $\bx, \by \in \RR^n$, we denote $\langle \bx, \by \rangle = \sum\limits^{n}_{i=1} x_i y_i$.

\section{Background}
\label{sec:background}

In this section, we recall some terminology and results in the reaction network theory.

\subsection{Euclidean Embedded Graphs and Mass-action Systems}
\label{sec:reaction_network}

In this subsection, we introduce the \emph{Euclidean embedded graph}, which is a directed graph in $\RR^n$,
and show how to define the \emph{mass-action system} from it.

\begin{definition}[\cite{craciun2015toric,craciun2020endotactic,craciun2019polynomial}]
\begin{enumerate}
\item[(a)] A \defi{reaction network} $G = (V, E)$, also called a \defi{Euclidean embedded graph (or E-graph)},
is a directed graph in $\RR^n$, where $V \subset \mathbb{R}^n$ represents a finite set of \defi{vertices} without isolated vertices, and $E \subseteq V \times V$ represents a finite set of \defi{edges} with no self-loops and at most one edge between a pair of ordered vertices.

\item[(b)] Let $V = \{ \by_1, \ldots, \by_m \}$. 
A directed edge $(\by_i, \by_j) \in E$, also called a \defi{reaction} in the network, is also denoted by $\by_i \to \by_j$, where $\by_i$ and $\by_j$ are called the \defi{source vertex} and \defi{target vertex} respectively.
Moreover, we define the \defi{reaction vector} associated with the reaction $\by_i \rightarrow \by_j$ to be $\by_j - \by_i \in\mathbb{R}^n$. 
\end{enumerate}
\end{definition}

\begin{definition} 
Let $G=(V, E)$ be an E-graph.
\begin{enumerate}[label=(\alph*)]
\item The set of vertices $V$ is partitioned by its connected components, also called \defi{linkage classes}.

\item A connected component $L \subseteq V$ is said to be \defi{strongly connected} if every edge is part of a directed cycle. 
Further, $G=(V, E)$ is \defi{weakly reversible} if all connected components are strongly connected.
\end{enumerate}
\end{definition}

\begin{example}
Figure~\ref{fig:e-graph} shows two reaction networks represented as E-graphs.

\begin{figure}[h!]
\centering
\includegraphics[scale=0.4]{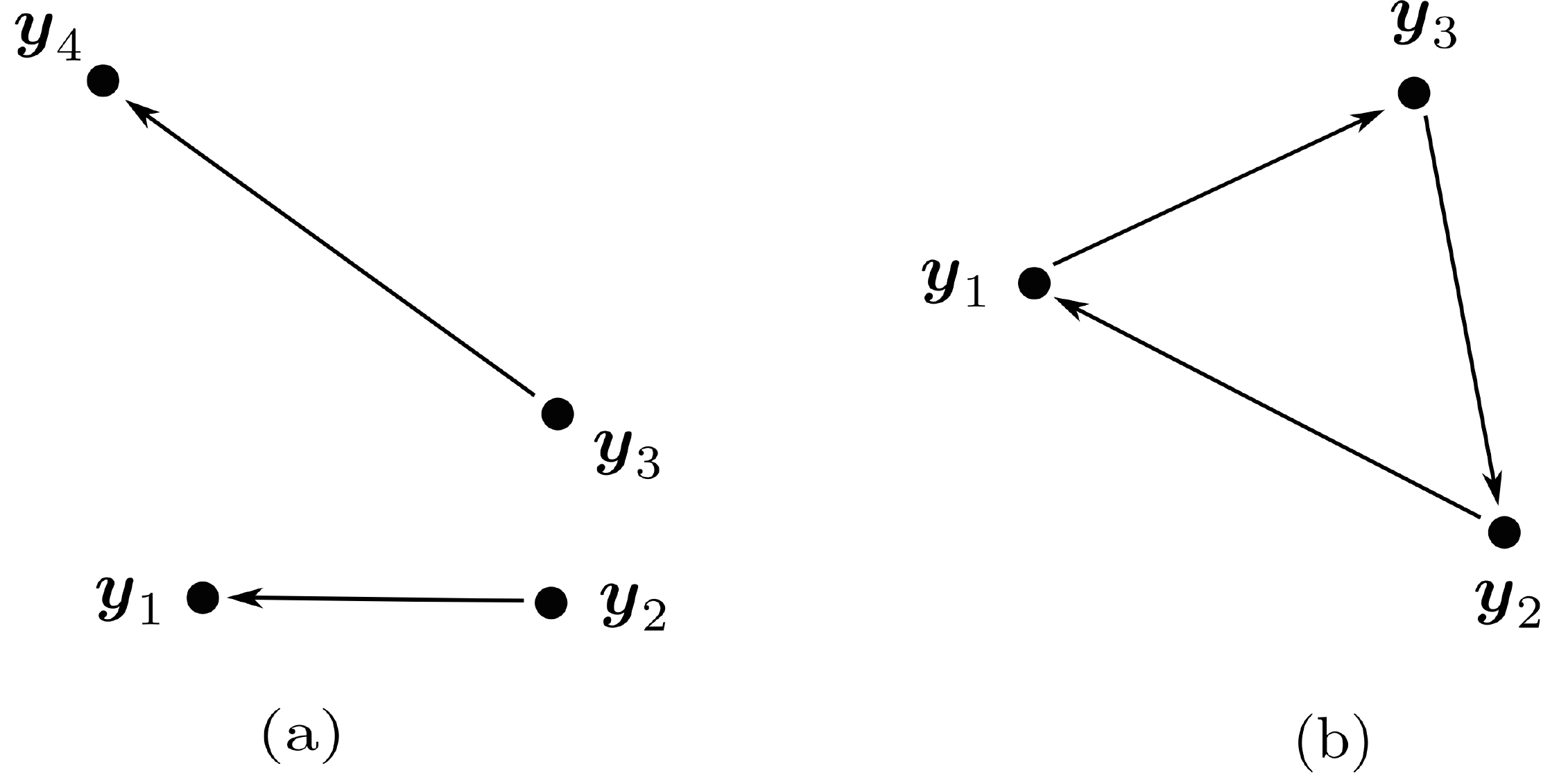}
\caption{\small (a) This reaction network consists of two linkage classes. (b) This reaction network is weakly reversible and contains one linkage class.}
\label{fig:e-graph}
\end{figure} 
\qed
\end{example}

\begin{definition}
\begin{enumerate}
\item[(a)] An E-graph $G=(V, E)$ is called a \defi{(directed) complete graph}, if for every pair of vertices $\by_i, \by_j \in V$, $\by_i \rightarrow \by_j \in E$. 

\item[(b)] 
Let $G_1 = (V_1, E_1)$ and $G_2 = (V_2, E_2)$ be two E-graphs. 
$G_1$ is called a \defi{subgraph of $G_2$} (denoted by $G_1 \subseteq G_2$), if $V_1 \subseteq V_2$ and $E_1 \subseteq E_2$. 
Further, we let $G_1 \sqsubseteq G_2$ denote that $G_1$ is a weakly reversible subgraph of $G_2$.
\end{enumerate}
\end{definition}

For any E-graph $G=(V, E)$, we can obtain a complete graph by connecting every pair of vertices in $V$, denoted by $G_c$, which is called the \defi{(directed) complete graph on $G$}.
We have $G \subseteq G_c$, and further if $G$ is weakly reversible, then $G \sqsubseteq G_c$.

\begin{definition}[\cite{yu2018mathematical,feinberg1979lectures}]
Let $G=(V, E)$ be an E-graph. 
Denote a \defi{reaction rate vector} by 
$$
\bk = (k_{\by_i \rightarrow \by_j})_{\by_i \rightarrow \by_j \in E} \in \mathbb{R}_{>0}^{E},
$$
where the positive number $k_{\by_i \rightarrow \by_j}$ or $k_{ij}$ is called the \defi{reaction rate constant} of the reaction $\by_i \rightarrow \by_j \in E$. The \defi{associated mass-action dynamical system} generated by $(G, \bk)$ is the dynamical system on  $\RR_{>0}^n$ given by
\begin{equation} \label{eq:mass_action}
 \frac{\mathrm{d} \bx}{\mathrm{d} t} 
= \sum_{\by_i \rightarrow \by_j \in E}k_{\by_i \rightarrow \by_j} \bx^{\by_i}(\by_j - \by_i).
\end{equation}
Moreover, we define the \defi{stoichiometric subspace} of $G$ as the span of the reaction vectors  of $G$, that is, 
\begin{equation}
\mSG = \spn \{ \by_j - \by_i: \by_i \rightarrow \by_j \in E \}.
\end{equation}
\end{definition}

Note that we set the domain of \eqref{eq:mass_action} to be $\mathbb{R}_{>0}^n$.  The system of ODEs does not allow
$\mathbb{R}_{>0}^n$ to be forward-invariant in general. But if we assume $V \subset \mathbb{Z}_{\geq 0}^n$, then the positive orthant $\mathbb{R}_{>0}^n$ is forward-invariant \cite{sontag2001structure}. 
Therefore, any solution to \eqref{eq:mass_action} with initial condition $\bx_0 \in \mathbb{R}_{>0}^n$ and $V \subset \mathbb{Z}_{\geq 0}^n$, is confined to the set $(\bx_0 + \mSG)\cap \mathbb{R}_{>0}^n$, which is  the \defi{affine invariant polyhedron} of $G$ at $\bx_0$.

\begin{definition} 
\label{def:mas_realizable}
Let $G=(V, E)$ be an E-graph.
Consider a dynamical system given by
\begin{equation} \label{eq:realization_ode}
\frac{\mathrm{d} \bx}{\mathrm{d} t} 
= \bf (\bx).
\end{equation}
This dynamical system is said to be \defi{$\RR$-realizable} (or has a \defi{$\RR$-realization}) on $G$, if there exists some $\bk \in \mathbb{R}^{E}$ such that
\begin{equation} \label{eq:realization}
\bf (\bx) =
\sum_{\by_i \rightarrow \by_j \in E}k_{\by_i \rightarrow \by_j} \bx^{\by_i}(\by_j - \by_i).
\end{equation}
Further, if $\bk \in \mathbb{R}^{E}_{>0}$, this dynamical system is said to be \defi{realizable} (or has a \defi{realization}) on $G$.
\end{definition}

\subsection{Complex-Balanced Systems and Flux Systems}
\label{sec:complex_balanced}

Here we focus on the \emph{complex-balanced systems} and \emph{complex-balanced flux systems}, which enjoy various graphic and dynamical properties. In addition, we build the \emph{toric locus} which serves to realize the complex-balanced systems on an E-graph.

\begin{definition} 
\label{def:cb_system}
Consider the mass-action system generated by $(G, \bk)$ in \eqref{eq:mass_action}.
A state $\bx^* \in \mathbb{R}_{>0}^n$ is called a \defi{positive steady state} if 
\begin{equation} 
\frac{\mathrm{d} \bx}{\mathrm{d} t} 
= \sum_{\by_i \rightarrow \by_j \in E}k_{\by_i \rightarrow \by_j} (\bx^*)^{\by_i}(\by_j - \by_i)
= \mathbf{0}.
\end{equation}
Further, a positive steady state $\bx^* \in \RR_{>0}^n$ is called a \defi{complex-balanced steady state}, if for every vertex $\by_0 \in V$,
\begin{equation} 
\sum_{\by_0 \to \by' \in E} k_{\by_0 \to \by'} (\bx^*)^{\by_0}
= \sum_{\by \to \by_0 \in E} k_{\by \to \by_0}
(\bx^*)^{\by}.
\end{equation}
We say the pair $(G, \bk)$ satisfies the \defi{complex-balanced conditions}, and the mass-action system generated by $(G, \bk)$ is called \defi{a complex-balanced system} or \defi{toric dynamical system}. 
\end{definition}

The following theorem illustrates some of the most essential properties of complex-balanced systems.   

\begin{theorem}[\cite{horn1972general}]
\label{thm:cb}
Let $(G, \bk)$ be a complex-balanced system, then 
\begin{enumerate}
\item[(a)] The E-graph $G=(V,E)$ is weakly reversible.

\item[(b)] All positive steady states are complex-balanced, and there is exactly one steady state within each invariant polyhedron. 

\item[(c)] Every complex-balanced steady state is asymptotically stable with respect to its invariant polyhedron. 
\end{enumerate}
\end{theorem}

\begin{definition} 
Let $G=(V, E)$ be an E-graph.
\begin{enumerate}

\item[(a)] Define the \defi{toric locus} on $G$ as
\begin{equation}
\mK (G) := \{ \bk \in \mathbb{R}_{>0}^{E} \ \big| \ \text{the mass-action system generated by } (G, \bk) \ \text{is toric} \}.
\end{equation}

\item[(b)] A dynamical system of the form 
\begin{equation} 
 \frac{\mathrm{d} \bx}{\mathrm{d} t} 
= \bf (\bx),
\end{equation}
is called \defi{disguised toric} on $G$, if it is realizable on $G$
for some $\bk \in \mK (G) \subseteq \mathbb{R}_{>0}^{E}$, i.e., it has a \defi{complex-balanced realization} on $G=(V, E)$.
\end{enumerate}
\end{definition}

In~\cite{CraciunDickensteinShiuSturmfels2009}, it is shown that the toric locus is a variety given by a binomial ideal, intersected with the positive orthant. The following theorem makes this precise.

\begin{theorem}[\cite{CraciunDickensteinShiuSturmfels2009}]
\label{thm:homeo}
Consider a weakly reversible E-graph $G = (V, E)$. Then $\mK(G)$ is a toric variety (up to a polynomial change of coordinates).
\end{theorem}

\begin{definition}
Let $G=(V, E)$ be an E-graph. 
Denote a \defi{flux vector} by $\bJ = (J_{\byi \to \byj})_{\byi \to \byj \in E} \in \RR_{>0}^E$, where $J_{\byi \to \byj}$ is called the \defi{flux} on the edge $\byi \to \byj \in E$. The \defi{associated flux system} on $\RR_{>0}^n$ generated by $(G, \bJ)$ is 
\begin{equation} \label{eq:flux}
 \frac{\mathrm{d} \bx}{\mathrm{d} t} 
= \sum_{\byi \to \byj \in E} J_{\byi \to \byj} 
(\byj - \byi).
\end{equation}
\end{definition}

\begin{definition}
\label{def:fluxVectors}
Consider the flux system generated by $(G, \bJ)$ in \eqref{eq:flux}. 
A flux vector $\bJ \in \RR_{>0}^E$ is called a \defi{steady flux vector} if 
\begin{equation}
\frac{\mathrm{d} \bx}{\mathrm{d} t} 
= \sum_{\byi \to \byj \in E} J_{\byi \to \byj} 
(\byj - \byi) = \mathbf{0}.
\end{equation}
Moreover, a steady flux vector $\bJ \in \RR_{>0}^E$ is called a \defi{complex-balanced flux vector}, if for every vertex $\by_0 \in V$,
\begin{equation}
\sum_{\by \to \by_0 \in E} J_{\by \to \by_0} 
= \sum_{\by_0 \to \by' \in E} J_{\by_0 \to \by'}.
\end{equation} 
We say the pair $(G, \bJ)$ forms a \defi{complex-balanced flux system}. In addition, we denote the set of all complex-balanced flux vectors on $G$ by
\begin{equation} 
\mathcal{J}(G):=
\{\bJ \in \RR_{>0}^{E} \mid \bJ  \text{ is a complex-balanced flux vector on $G$}\}.
\end{equation}
\end{definition}

Analogous to complex-balanced systems, complex-balanced flux systems also have a connection with E-graphs.

\begin{lemma}[\cite{craciun2020efficient}]
Every E-graph which permits a complex-balanced flux system is weakly reversible.
On the other side, every E-graph which is weakly reversible permits complex-balanced flux systems.
\end{lemma}

Moreover, when flux vectors are constructed under mass-action kinetics, complex-balanced systems can be linked with complex-balanced flux systems.

\begin{lemma}[\cite{craciun2020efficient}]
Suppose $(G, \bk)$ is a complex-balanced system with a steady state $\bx^* \in \RR^n_{>0}$. Consider the flux vector $\bJ = (J_{\byi \to \byj})_{\byi \to \byj \in E}$ with $J_{\byi \to \byj} = k_{\byi \to \byj} (\bx^*)^{\by_i}$, then the pair $(G, \bJ)$ forms a complex-balanced flux system.
\end{lemma}

As a direct consequence, we obtain the following remark.

\begin{remark}
\label{rmk:graph_KJ}
Let $G=(V, E)$ be an E-graph.
\begin{enumerate}
\item[(a)] If $G=(V,E)$ is weakly reversible, then $\mK (G) \neq \emptyset$ and $\mathcal{J}(G) \neq \emptyset$.

\item[(b)] If $G=(V,E)$ is not weakly reversible, then $\mK (G) = \mathcal{J}(G) = \emptyset$.
\end{enumerate}
\end{remark}

\subsection{Dynamical Equivalence}
\label{sec:dynamical_equivalence}

Under mass-action kinetics, different reaction networks can give rise to the same dynamical system. In this subsection, we introduce the \emph{dynamical equivalence} under which two mass-action systems share the same associated dynamics.

\begin{definition}[\cite{horn1972general,craciun2008identifiability}]
\label{def:de}
Two mass-action systems $(G, \bk) = (V, E, \bk)$ and $(G', \bk') = (V', E', \bk')$ are said to be \defi{dynamically equivalent}, if for every vertex\footnote{\label{footnote1} Note that when $\by_0 \not\in V$ or $\by_0 \not\in V'$, the side is considered as an empty sum} $\by_0 \in V \cup V'$,
\begin{equation}
\label{eq:DE}
\sum_{\by_0 \to \by \in E} k_{\by_0  \to \by} (\by - \by_0) 
= \sum_{\by_0 \to \by' \in E'} k'_{\by_0  \to \by'}  (\by' - \by_0).
\end{equation}
We let $(G, \bk) \sim (G', \bk')$ denote that 
two systems $(G, \bk)$ and $(G', \bk')$ are dynamically equivalent.
\end{definition} 


\begin{example}
Figure~\ref{fig:dyn_equiv} gives an example of two dynamically equivalent mass-action systems. 

\begin{figure}[!ht]
\centering
\includegraphics[scale=0.4]{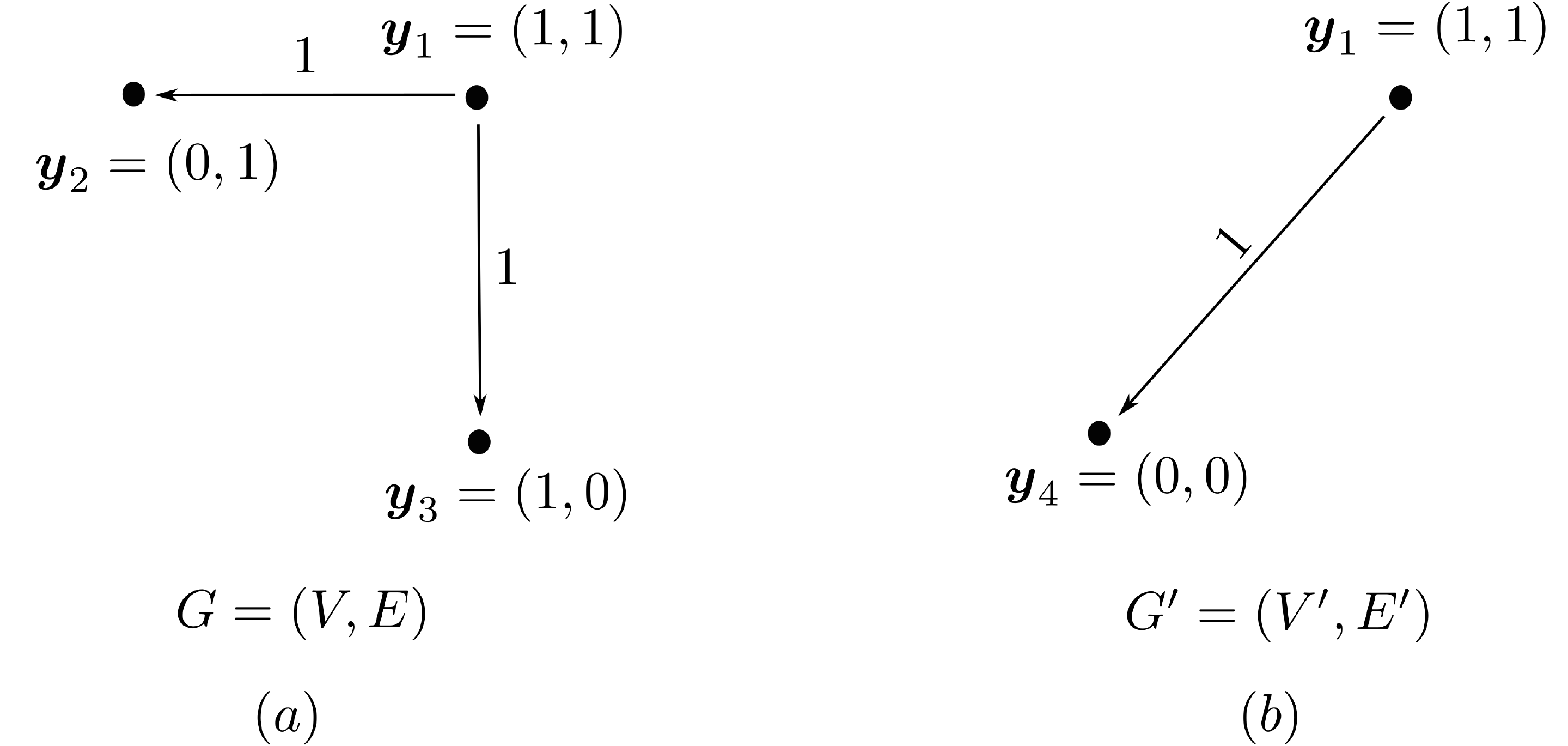}
\caption{\small The mass-action systems in (a) and (b) are dynamically equivalent.}
\label{fig:dyn_equiv}
\end{figure} 

Since $\by_1$ is the only source vertex in systems $G$ and $G'$, it suffices to check whether two systems satisfy Equation~\eqref{eq:DE} on the vertex $\by_1$.
For the system $(G, \bk)$, we have
\begin{equation}
\sum_{\by_1 \rightarrow \by \in E} k_{\by_1 \rightarrow \by} (\by - \by_1) =
k_{12} \begin{pmatrix} -1 \\ 0 \end{pmatrix}
+ k_{13} \begin{pmatrix} 0 \\ -1 \end{pmatrix}
= \begin{pmatrix} -1 \\ -1 \end{pmatrix}.
\end{equation}
For the system $(G', \bk')$, we have
\begin{equation}
\sum_{\by_1 \rightarrow \by' \in E'} k_{\by_1 \rightarrow \by'} (\by' - \by_1) =
k'_{12} \begin{pmatrix} -1 \\ -1 \end{pmatrix}
= \begin{pmatrix} -1 \\ -1 \end{pmatrix}.
\end{equation}
Hence, two systems $(G, \bk)$ and $(G', \bk')$ are dynamically equivalent.
\qed
\end{example}

\begin{remark}
Follow Definition \ref{def:de}, two mass-action systems $(G, \bk)$ and $(G', \bk')$ are dynamically equivalent if and only if for all $\bx \in \RR_{>0}^{n}$,
\begin{equation}
\label{eq:eqDE}
\sum_{\by_1 \to \by_2 \in E} k_{\by_1  \to \by_2} \bx^{\by_1} (\by_2 - \by_1) 
= \sum_{\by_1' \to \by_2' \in E'} k'_{\by'_1  \to \by'_2} \bx^{\by'_1} (\by'_2 - \by'_1).
\end{equation}
\end{remark}

\begin{remark}
Suppose $(G, \bk)$ and $(G', \bk')$ are two dynamically equivalent mass-action systems. Then $(G, \bk)$ is realizable on $G'$ and $(G', \bk')$ is realizable on $G$.
\end{remark}

\begin{definition} 
\label{def:d0}
Let $G=(V, E)$ be an E-graph and let $\bd = (d_{\by \to \by'})_{\by \to \by' \in E} \in \RR^{|E|}$. We define the set $\mD(G)$ as
\begin{equation} \notag
\mD (G) :=
\{\bd \in \RR^{|E|} \, \bigg| \, \sum_{\by_0 \to \by \in E} d_{\by_0  \to \by} (\by - \by_0) = \mathbf{0} \ \text{for every vertex } \by_0 \in V
\}.
\end{equation}
\end{definition} 

\begin{definition} 
\label{def:FE}
Two flux systems $(G,\bJ) = (V, E, \bJ)$ and $(G', \bJ') = (V', E',\bJ')$ are said to be \defi{flux equivalent}, if for every vertex\footref{footnote1} $\by_0 \in V \cup V'$
\begin{equation}
\sum_{\by_0 \to \by \in E} J_{\by_0 \to \by} (\by - \by_0) 
= \sum_{\by'_0 \to \by' \in E'} J'_{\by_0 \to \by'} (\by' - \by'_0).
\end{equation}
We let $(G, \bJ) \sim (G', \bJ')$ denote that two systems $(G, \bJ)$ and $(G', \bJ')$ are flux equivalent. 
\end{definition} 

\begin{definition} 
\label{def:j0}
Let $G=(V, E)$ be an E-graph and let $\bJ = ({J}_{\byi \to \byj})_{\byi \to \byj \in E} \in \RR^E$.
We define the set $\eJ (G)$ as
\begin{equation} \label{eq:J_0}
\eJ (G): =
\{{\bJ} \in \mD (G) \, \bigg| \, \sum_{\by \to \by_0 \in E} {J}_{\by \to \by_0} 
= \sum_{\by_0 \to \by' \in E} {J}_{\by_0 \to \by'} \ \text{for every vertex } \by_0 \in V
\}.
\end{equation}
\end{definition} 

\begin{example}
\label{ex:D0J0G2G3}
Figure~\ref{fig:graphs_D_J} illustrates two E-graphs $G_1 = (V_1, E_1)$ and $G_2 = (V_2, E_2)$. Here we compute $\mD$ and $\eJ$ for both E-graphs. 

\begin{figure}[H]
\centering
\includegraphics[scale=0.4]{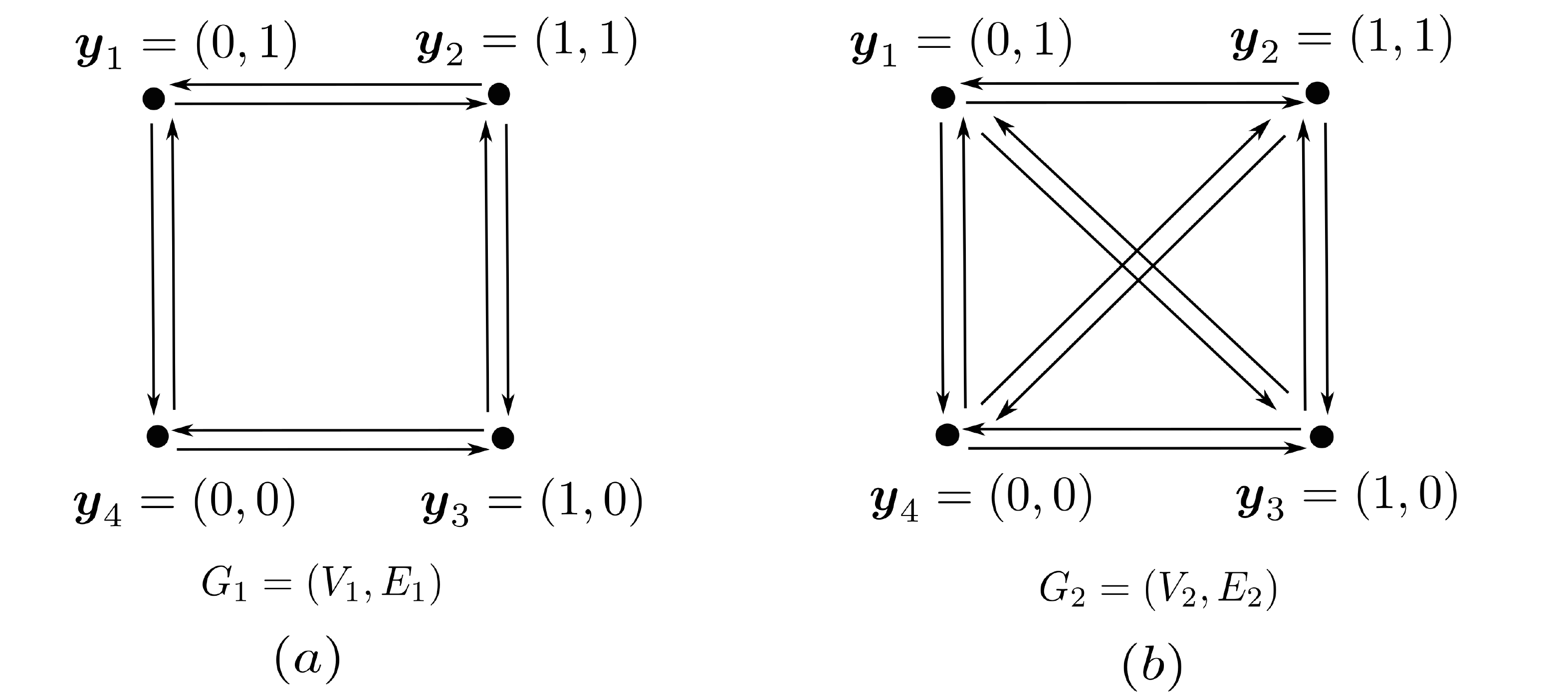}
\caption{Two weakly reversible E-graphs $G_1$ and $G_2$.}
\label{fig:graphs_D_J}
\end{figure} 

\begin{enumerate}
\item[(a)] For E-graph $G_1$, we consider every vertex $\by \in V_1$ and compute that all corresponding edges $\by \to \by' \in E_1$ are linearly independent. Thus, we derive that
\[
\mD = \eJ = \{ \mathbf{0} \}.
\]

\item[(b)] For E-graph $G_2$, we start with the vertex $\by_1 \in V_1$. Note that the reaction vectors $\{ \by_i - \by_1 \}$ for $i = 2,3,4$ are linearly dependent as follows:
\[
\by_3 - \by_1 = (\by_2 - \by_1) + (\by_4 - \by_1).
\]
Thus, we find a vector in $\mD(G_2)$ and denote it by $\bv_1$, such that
\begin{equation} \notag
\bv_{1, r} = 
\begin{cases}
1, & \text{ if } r = \by_1 \to \by_2 \text{ or } \by_1 \to \by_4, \\
-1, & \text{ if } r = \by_1 \to \by_3, \\
0, & \text{ Otherwise}.
\end{cases}
\end{equation}
For the vertex $\by_2 \in V_1$ and the corresponding reaction vectors $\{ \by_i - \by_2 \}$ for $i = 1,3,4$, we have
\[
\by_4 - \by_2 = (\by_1 - \by_2) + (\by_3 - \by_2).
\]
Thus we find another vector $\bv_2 \in \mD(G_2)$, such that
\begin{equation} \notag
\bv_{2, r} = 
\begin{cases}
1, & \text{ if } r = \by_2 \to \by_1 \text{ or } \by_2 \to \by_3, \\
-1, & \text{ if } r = \by_2 \to \by_4, \\
0, & \text{ Otherwise}.
\end{cases}
\end{equation}

Similarly, for vertices $\by_3, \by_4 \in V_1$, we find two vectors $\bv_3, \bv_4 \in \mD(G_2)$, such that
\begin{equation} \notag
\bv_{3, r} = 
\begin{cases}
1, & \text{ if } r = \by_3 \to \by_2 \text{ or } \by_3 \to \by_4, \\
-1, & \text{ if } r = \by_3 \to \by_1, \\
0, & \text{ Otherwise}.
\end{cases} 
\text{ and } \
\bv_{4, r} = 
\begin{cases}
1, & \text{ if } r = \by_4 \to \by_1 \text{ or } \by_4 \to \by_3, \\
-1, & \text{ if } r = \by_4 \to \by_2, \\
0, & \text{ Otherwise}.
\end{cases}
\end{equation}
Therefore, we derive that 
\[
\mD(G_2) = \spn \{ \bv_1, \bv_2, \bv_3, \bv_4 \},
\]
and $\dim (\mD(G_2)) =  4$.

\medskip

Next, we compute the flux changes on every vertex after applying flux vectors $\{ \bv_i\}^{4}_{i=1}$ in $\mD(G_2)$. Then we find that
\[
\bv_1 + \bv_2, \ \bv_1 + \bv_3, \ \bv_1 - \bv_4 \in \eJ (G_2),
\]
and these three flux vectors are linearly independent. 
Further, $\bv_1 \notin \eJ (G_2)$ because of the flux change on $\by_4$ is not balanced to zero. 
On the other hand, note that $\eJ (G_2) \subset \mD (G_2)$ and $\dim (\mD(G_2)) =  4$.
Therefore, we conclude that 
\[
\eJ (G_2) = \spn \{ \bv_1 + \bv_2, \ \bv_1 + \bv_3, \ \bv_1 - \bv_4 \},
\]
and $\dim (\eJ(G_2)) =  3$.
\end{enumerate}
\qed
\end{example}

\begin{lemma}
\label{lem:d0}
Consider two mass-action systems $(G, \bk)$ and $(G, \bk')$, then $(G, \bk) \sim (G, \bk')$ if and only if $\bk' - \bk \in \mD (G)$.    
\end{lemma}

\begin{proof}
Suppose that $(G, \bk) \sim (G, \bk')$. 
Then for every vertex $\by_0 \in V$, we have
\begin{equation} \notag
\sum_{\by_0 \to \by \in E} k_{\by_0  \to \by} (\by - \by_0) 
= \sum_{\by_0 \to \by' \in E} k'_{\by_0  \to \by}  (\by - \by_0).
\end{equation}
This is equivalent to 
\begin{equation} \notag
\sum_{\by_0 \to \by \in E} (k_{\by_0  \to \by} - k'_{\by_0  \to \by} ) (\by - \by_0) 
= \mathbf{0}.
\end{equation}
This implies that
\begin{equation} \notag
\bk' - \bk = (k'_{\by_0  \to \by} - k_{\by_0  \to \by})_{\by_0  \to \by \in E} \in \mD (G).
\end{equation}
Reversing the steps above, it is clear that given $\bk' - \bk \in \mD (G)$, we can obtain $(G, \bk) \sim (G, \bk')$.
\end{proof}

\begin{lemma}
\label{lem:j0}
Consider two flux systems $(G, \bJ)$ and $(G, \bJ')$, then
\begin{enumerate}
\item[(a)] $(G, \bJ) \sim (G, \bJ')$ if and only if $\bJ' - \bJ \in \mD (G)$.

\item[(b)] Further, if $(G, \bJ)$ and $(G, \bJ')$ are both complex-balanced flux systems, then $(G, \bJ) \sim (G, \bJ')$ if and only if $\bJ' - \bJ \in \eJ(G)$.
\end{enumerate} 
\end{lemma}

\begin{proof}
\begin{enumerate}
\item[(a)] This follows directly from Definition \ref{def:j0} and Lemma \ref{lem:d0}.

\item[(b)] Since both $(G, \bJ)$ and $(G, \bJ')$ are complex-balanced flux systems, for every vertex $\by_0 \in V$
\begin{equation} \notag
\sum_{\by \to \by_0 \in E} J_{\by \to \by_0} 
= \sum_{\by_0 \to \by' \in E} J_{\by_0 \to \by'}
\ \text{ and } \ 
\sum_{\by \to \by_0 \in E} J'_{\by \to \by_0} 
= \sum_{\by_0 \to \by' \in E} J'_{\by_0 \to \by'}.
\end{equation}
Thus, we have
\begin{equation} \notag
\sum_{\by \to \by_0 \in E} (J_{\by \to \by_0} - J'_{\by \to \by_0})
= \sum_{\by_0 \to \by' \in E} (J_{\by_0 \to \by'} - J'_{\by_0 \to \by'}).
\end{equation}
Suppose $(G, \bJ) \sim (G, \bJ')$, we derive that
\begin{equation} \notag
\bJ' - \bJ = (J'_{\by_0  \to \by} - J_{\by_0  \to \by})_{\by_0  \to \by \in E} \in \eJ(G)
\end{equation}
and vice versa.
\end{enumerate}
\end{proof}

\begin{remark}
\label{rmk:linear_subspace}
Let $G=(V, E)$ be an E-graph. Then both
$\mD (G)$ and $\eJ(G)$ are linear subspaces of $\RR^E$.
\end{remark}

\medskip

The following proposition shows the connection between dynamical equivalence and flux equivalence when the flux vector is constructed under mass-action kinetics.

\begin{proposition}[\cite{craciun2020efficient}]
\label{prop:craciun2020efficient}
Let $(G, \bk)$ and $(G', \bk')$ be two mass-action systems and let $\bx \in \RR_{>0}^n$. Define the flux vector $\bJ (\bx) = (J_{\by \to \by'})_{\by \to \by' \in E}$ on $G$, such that for every $\by \to \by' \in E$
\begin{equation}
J_{\by \to \by'} = k_{\by \to \by'} \bx^{\by}.
\end{equation}
Further, define the flux vector $\bJ' (\bx) = (J'_{\by \to \by'})_{\by \to \by' \in E'}$ on $G'$, such that for every $\by \to \by' \in E$
\begin{equation}
J'_{\by \to \by'} = k'_{\by \to \by'} \bx^{\by}.
\end{equation} 
Then the following are equivalent:
\begin{enumerate}

\item[(a)] the mass-action systems $(G, \bk)$ and $(G', \bk')$ are dynamically equivalent.

\item[(b)] the flux systems $(G, \bJ(\bx))$ and $(G', \bJ'(\bx))$ are flux equivalent for all $\bx \in \RR_{>0}^n$.

\item[(c)] the flux systems $(G, \bJ(\bx))$ and $(G', \bJ'(\bx))$ are flux equivalent for some $\bx \in \RR_{>0}^n$
\end{enumerate}
\end{proposition}

\subsection{\texorpdfstring{$\RR$}{R}-Disguised Toric Locus}
\label{sec:disguised_locus}

In this subsection, we define the \emph{disguised toric locus and the $\RR$-disguised toric locus} to collect the reaction rates that allow complex-balanced realizations under dynamical equivalence.

\begin{definition} 
\label{def:de_realizable}
Let $G =(V,E)$ and $\hat{G} =(\hat{V}, \hat{E})$ be two E-graphs.
\begin{enumerate}


\item[(a)] Define the set $\mK_{\RR}(\hat{G}, G)$ as 
\begin{equation} \notag
\mK_{\RR}(\hat{G}, G) := \{ \hat{\bk} \in \mK (\hat{G}) \ \big| \ \text{the mass-action system } (\hat{G}, \hat{\bk}) \ \text{is $\RR$-realizable on } G \}.
\end{equation}

%

\item[(b)] Define the set $\dK(G, \hat{G})$ as
\begin{equation} \notag
\dK(G, \hat{G}) := \{ \bk \in \mathbb{R}^{E} \ \big| \ \text{the dynamical system} \ (G, \bk) \ \text{is disguised toric on } \hat{G} \},
\end{equation}
where the dynamical system generated by $(G, \bk)$ is referred to \eqref{eq:realization_ode}-\eqref{eq:realization}. Note that $\hat{\bk}$ may have non-positive components.


\item[(c)] Define the \defi{$\RR$-disguised toric locus} of $G$ as
\begin{equation} \notag
\dK(G) = \bigcup_{\tilde{G} \sqsubseteq G_{c}} \ \dK(G, \tilde{G}),
\end{equation}
where $\tilde{G}\sqsubseteq G_{c}$ represents that
$\tilde{G}$ is a weakly reversible subgraph of $ G_{c}$.
\end{enumerate}

\end{definition}


\begin{remark}
In general, of course, we need $\dK (G)$ to include $\dK(G, G')$ for {\em any} weakly reversible E-graph $G'$ (i.e., not just for $G' \sqsubseteq G_{c}$). On the other hand, due to results in~\cite{craciun2020efficient}, it turns out that, if a dynamical system generated by $G$ can be realized as toric by some $G'$, then there exists $G'' \sqsubseteq G_{c}$ that {\em also} can give rise to a toric realization of that dynamical system. Therefore, the above assumption that $G_i \sqsubseteq G_{c}$ still leads to the correct definition of $\dK (G)$. 
\end{remark}

\begin{remark}
Note that the definition of the disguised toric locus (denoted by $\mathcal{K}_{\rm disg}(G)$) is similar to Definition~\ref{def:de_realizable}, with the rate constants allowed to take only positive values, i.e., $\bk \in \mathbb{R}^{E}_{>0}$.
\end{remark}

\begin{definition}
\label{def:flux_realizable}
Let $(G', \bJ')$ be a flux system. We say it is 
\defi{$\RR$-realizable} on $G$ if there exists some $\bJ \in \mathbb{R}^{E}$, such that for every vertex\footref{footnote1} $\by_0 \in V \cup V'$,
\begin{equation} \notag
\sum_{\by_0 \to \byj \in E} J_{\by_0 \to \byj} 
(\byj - \by_0) 
= \sum_{\by_0 \to \bypj \in E'} J'_{\by_0 \to \bypj} 
(\bypj - \by_0).
\end{equation}
Further, we define the set $\mJ (G', G)$ as
\begin{equation} \notag
\mJ (G', G) := \{ \bJ' \in \mJ (G') \ \big| \ \text{the flux system } (G', \bJ') \ \text{is $\RR$-realizable on } G \}.
\end{equation}



\end{definition}

\begin{remark} \label{rmk:mJ_dK}
Let $G_1 = (V_1, E_1)$ be a weakly reversible E-graph and let $G = (V, E)$ be an E-graph. From Definition \ref{def:de_realizable}, it follows that $\dK (G, G_1)$ is empty if and only if $\mK_{\RR} (G_1, G)$ is empty. Meanwhile, from Proposition \ref{prop:craciun2020efficient}, it follows that $\dK (G, G_1)$ is empty if and only if $\mJ(G_1, G)$ is empty.
\end{remark}

Recall that a set $X$ is a \textbf{polyhedral cone} if $X = \{\bx \in \RR^n: M \bx \geq \textbf{0} \text{ for some matrix } M \}$. 
A \textbf{finite cone} is the conic combination of finitely many vectors.
The Minkowski-Weyl theorem \cite{2016david} states that every polyhedral cone is a finite cone and vice-versa.

\begin{lemma}
\label{lem:j_g1_g_cone}
Let $G_1 = (V_1, E_1)$ be a weakly reversible E-graph and let $G = (V, E)$ be an E-graph. Then there exists a set of vectors $\{ \bv_1, \bv_2, \ldots, \bv_k \} \subset \RR^{|E_1|}$, such that 
\begin{equation} \label{j_g1_g_generator}
\mJ (G_1, G) = \{ a_1 \bv_1 + \cdots a_k \bv_k \ | \ a_i \in \RR_{>0}, \bv_i \in \RR^{|E_1|} \},
\end{equation} 
and $\dim (\mJ (G_1, G)) =\dim ( \spn \{ \bv_1, \bv_2, \ldots, \bv_k \} )$.
Moreover, if $\mJ (G_1, G) \neq \emptyset$, then
\[
\eJ(G_1) \subseteq \spn \{ \bv_1, \bv_2, \ldots, \bv_k \}.
\]
\end{lemma}


\begin{proof}
If $\mJ (G_1, G) = \emptyset$, it is clear that $\dim (\mJ (G_1, G)) = 0$ and we conclude the lemma. Now suppose $\mJ (G_1, G) \neq \emptyset$. Consider any flux vector $\bJ_1 = (J_{1, \by \to \by'})_{\by \to \by' \in E_1} \in \mJ(G_1, G)$. From Definition \ref{def:flux_realizable}, there exists the flux vector $\bJ = (J_{\by \to \by'})_{\by \to \by' \in E} \in \RR^{|E|}$, such that
\[
(G, \bJ) \sim (G_1, \bJ_1).
\]
Thus, for every vertex $\by_0 \in V_1 \ / \ V$, 
\begin{equation} \label{eq:semi_j_g1g_1}
\sum_{\by_0 \to \by' \in E_1} J_{1, \by_0  \to \by'}  (\by' - \by_0) = \mathbf{0}.
\end{equation}
For every vertex $\by_0 \in V \cap V_1$, 
\begin{equation} \notag
 \sum_{\by_0 \to \by' \in E_1} J_{1, \by_0  \to \by'}  (\by' - \by_0)
= \sum_{\by_0 \to \by \in E} J_{\by_0  \to \by} (\by - \by_0),
\end{equation}
which is equivalent to that for every vertex $\by_0 \in V \cap V_1$,
\begin{equation} \label{eq:semi_j_g1g_2}
 \sum_{\by_0 \to \by' \in E_1} J_{1, \by_0  \to \by'}  (\by' - \by_0) \in \spn \{ \by - \by_0 \}_{\by_0 \to \by \in E}.
\end{equation}
Further, since $\bJ_1 \in \mJ (G_1)$ we have for every vertex $\by_0 \in V_1$,
\begin{equation} \label{eq:semi_j_g1g_3}
 \sum_{\by \to \by_0 \in E_1} J_{1, \by \to \by_0} 
= \sum_{\by_0 \to \by' \in E_1} J_{1, \by_0 \to \by'}.
\end{equation}

Now consider the set of flux vectors as follows:
\begin{equation} \notag
\widetilde{\mJ} (G_1, G) := \{ \bJ_1 \in \RR^{|E_1|} \ \big| \ \bJ_1 
\text{ satisfies Equations}~\eqref{eq:semi_j_g1g_1}, \eqref{eq:semi_j_g1g_2}
\text{ and }
\eqref{eq:semi_j_g1g_3} 
\}.
\end{equation}
It is clear that $\widetilde{\mJ} (G_1, G)$ is a linear subspace of $\RR^{|E_1|}$, thus $\widetilde{\mJ} (G_1, G)$ is a polyhedral cone.
Hence there exists some matrix $\widetilde{M}$, such that
$\widetilde{\mJ} (G_1, G) = \{\bx \in \RR^{|E_1|}: \widetilde{M} \bx \geq \textbf{0} \}$. 
Here we set a new matrix $M$ as 
\[
M =
\left[\begin{array}{c}
\widetilde{M} \\ \hdashline[2pt/2pt] 
\bI_{E_1}
\end{array}
\right],
\]
where $\bI_{E_1}$ is an identity matrix of size $E_1$.
Then we get
\[
\widetilde{\mJ} (G_1, G) \cap \RR^n_{\geq 0} = \{\bx \in \RR^{|E_1|}: M \bx \geq \textbf{0} \}.
\]
By Minkowski-Weyl theorem, there is a set of vectors $\{ \bv_1, \bv_2, \ldots, \bv_k \}$ such that 
\begin{equation} \notag
\widetilde{\mJ} (G_1, G) \cap \RR^n_{\geq 0} = \{ a_1 \bv_1 + \cdots a_k \bv_k \ | \ a_i \in \RR_{\geq 0}, \bv_i \in \RR^{|E_1|} \}.
\end{equation} 
Note that $\mJ (G_1, G) = \widetilde{\mJ} (G_1, G) \cap \RR^n_{>0}$ is open.
Hence we derive
\[
\mJ (G_1, G) = \{ a_1 \bv_1 + \cdots a_k \bv_k \ | \ a_i \in \RR_{>0}, \bv_i \in \RR^{|E_1|}\},
\]
and we prove \eqref{j_g1_g_generator}.

\medskip

Next, if $\mJ (G_1, G) \neq \emptyset$, consider any flux vector $\bJ_1 \in \mJ (G_1, G) \subset \RR^{|E_1|}_{>0}$. Then there must exist sufficiently small number $\varepsilon > 0$, such that
\begin{equation} \label{j_g1_g_generator_2}
\bJ_1 + b_1 \bv_1 + \cdots + b_k \bv_k \in \RR^{|E_1|}_{>0},
\ \text{ for any }
1 \leq i \leq k \text{ and }
|b_i| \leq \varepsilon.
\end{equation}
Using \eqref{j_g1_g_generator}, we derive that $\bJ_1 + \sum\limits^k_{i=1} b_i \bv_i \in \mJ (G_1, G)$ and every neighbourhood of $\bJ_1$ is in the form of \eqref{j_g1_g_generator_2}. Hence, we get
\begin{equation} \label{j_g1_g_generator_dim}
\dim (\mJ (G_1, G)) =\dim ( \spn \{ \bv_1, \bv_2, \ldots, \bv_k \} ).
\end{equation}
Moreover, for any $\hbJ \in \eJ(G_1)$ we can also find sufficiently small number $\hat{\varepsilon} > 0$, such that
\[
\bJ_1 + \hat{b} \hbJ \in \RR^n_{>0},
\ \text{ for any }
|\hat{b}| \leq \hat{\varepsilon}.
\]
Lemma \ref{lem:j0} shows that $\bJ_1 + \hat{b} \hbJ \in \mJ (G_1, G)$.
Using \eqref{j_g1_g_generator_dim}, we conclude that
\[
\eJ(G_1) \subset \spn \{ \bv_1, \bv_2, \ldots, \bv_k \}.
\]
\end{proof}

\section{Main results}
\label{sec:main_result}

In this section, we present the main result of this paper where we give a lower bound on the \emph{dimension of an $\RR$-disguised toric locus}.

\paragraph{Notation.}
For simplicity, throughout this section, we abuse the notation in part $(a)$ and introduce some notations in part $(b)$ as follows:
\begin{enumerate}
\item[(a)] Given $\bk = (k_{\by_i \rightarrow \by_j})_{\by_i \rightarrow \by_j \in E} \in \mathbb{R}^{E}$, we consider the associated ``mass-action" dynamical system generated by $(G, \bk)$ as \eqref{eq:realization_ode}-\eqref{eq:realization}. Note that $\bk$ may have non-positive components; if so, these systems are not the usual mass-action systems.

\item[(b)] We consider $G = (V, E)$ to be an E-graph.
Let $b$ denote the dimension of the linear subspace $\mD(G)$, and denote an orthonormal basis of $\mD(G)$ by
\[
\{\bB_1, \bB_2, \ldots, \bB_b\}.
\]
Moreover, we consider $G_1 = (V_1, E_1)$ to be a weakly reversible E-graph.
Let $a$ denote the dimension of the subspace $\eJ(G_1)$, and denote an orthonormal basis of $\eJ(G_1)$ by 
\[
\{\bA_1, \bA_2, \ldots, \bA_a \}.
\]
\end{enumerate}



\subsection{An Injective and Continuous Map to the \texorpdfstring{$\RR$}{R}-Disguised Toric Locus}
\label{sec:psi_function}

Recall that given two E-graphs $G$ and $G_1$, $\dK (G, G_1)$ is the set of reaction rates in $G$ for which there exists a complex-balanced realization in $G_1$. Here we introduce the function $\Psi$ (see Definition \ref{def:psi}) to build a connection between $\dK (G, G_1)$ and $\mJ(G_1,G)$.






 



\begin{definition} \label{def:Q}
Consider a weakly reversible E-graph $G_1 = (V_1, E_1)$. For any E-graph $G = (V, E)$, define the map 
\begin{equation} \notag
\psi: \mJ(G_1,G) \rightarrow \mathbb{R}^a 
\end{equation}
such that for $\bJ \in \mJ(G_1,G)$, 
\begin{equation} \notag
\psi(\bJ) = (\bJ \bA_1, \bJ \bA_2, \ldots, \bJ \bA_a).
\end{equation}
Moreover, we define the set $Q$ as
\[
Q := \psi (\mJ(G_1,G)) = \{\psi(\bJ) \ | \ \bJ \in \mJ(G_1,G) \}.
\]
\end{definition}

\begin{lemma}
\label{lem:Q_open}
Consider a weakly reversible E-graph $G_1 = (V_1, E_1)$. For any E-graph $G = (V, E)$, suppose $\mJ(G_1,G) \neq \emptyset$. Then the set 
$Q$
in \eqref{def:Q} is an open set in $\RR^a$.
\end{lemma}

\begin{proof} 
Consider a complex-balanced flux vector $\bJ \in \mJ(G_1,G) \subset \RR^{|E_1|}_{>0}$.
Recall $\{\bA_i\}^a_{i=1}$ forms an orthonormal basis of the subspace $\eJ(G_1)$. Given a vector $\br = (r_1, r_2, \ldots, r_a)$, we consider the following flux vector:
\[
\bJ_{\br} = \bJ + \sum\limits^a_{i=1} r_i A_i.
\]
Since $\{\bA_i\}^a_{i=1}$ are unit vectors and $\bJ \in \RR^{|E_1|}_{>0}$, there exists a sufficiently small positive number $\epsilon > 0$, such that for any $|\br| \leq \epsilon$, $\bJ_{\br} \in \RR^{|E_1|}_{>0}$.
Using Lemma \ref{lem:j_g1_g_cone}, we get
\begin{equation} \label{eq:jr}
(G, \bJ) \sim (G, \bJ_{\br}) 
\ \text{ and } \ 
\bJ_{\br} \in \mJ(G_1,G).
\end{equation}

From Definition~\ref{def:Q}, for any $\bq^* \in Q$ there exists a flux vector $\bJ^* \in \mJ(G_1,G)$.
From \eqref{eq:jr}, there exists a  positive real number $\epsilon^* > 0$ such that
\[
\bJ^*_{\br} \in \mJ(G_1,G),
\ \text{ for any }
|\br| \leq \epsilon^*.
\]
From $\{\bA_i\}^a_{i=1}$ being an orthonormal basis, for any $|\br| \leq \epsilon^*$
\[
\psi (\bJ^*_{\br}) = \psi (\bJ^*) + \br = \bq^* + \br \in Q.
\]
Therefore, we conclude that the set $Q$ is open.
\end{proof}

\begin{definition} \label{def:psi}
Given a weakly reversible E-graph $G_1 = (V_1, E_1)$ with its stoichiometric subspace $\mS_{G_1}$.
Consider an E-graph $G = (V, E)$ and $\bx_0\in\mathbb{R}^n_{>0}$, define the map 
\begin{equation} \label{eq:psi}
\Psi: \mJ(G_1,G)\times [(\bx_0 + \mS_{G_1} )\cap\mathbb{R}^n_{>0}] \times \mathbb{R}^b \rightarrow \dK(G,G_1) \times Q
\end{equation}
such that for $(\bJ, \bx, \bp) \in \mJ(G_1,G)\times [(\bx_0 + \mS_{G_1} )\cap\mathbb{R}^n_{>0}] \times \mathbb{R}^b$, 
\begin{equation} \notag
\Psi(\bJ,\bx, \bp) := (\bk, \bq),
\end{equation}
where 
\begin{equation} \label{def:psi_k}
(G, \bk) \sim (G_1, \bk_1) \ \text{ with } \ k_{1, \by\rightarrow \by'} = \frac{J_{\by\rightarrow \by'}}{{\bx}^{\by}}.
\end{equation} 
and
\begin{equation} \label{def:psi_kq}
\bp = (\bk \bB_1, \bk \bB_2, \ldots, \bk \bB_b), \ \
\bq = (\bJ \bA_1, \bJ \bA_2, \ldots, \bJ \bA_a).
\end{equation}
\end{definition}

Recall Remark \ref{rmk:mJ_dK}, $\dK (G, G_1)$ is empty if and only if $\mJ(G_1, G)$ is empty. 
If $\mJ(G_1, G) = \dK (G, G_1) = \emptyset$, it is clear that the map $\Psi$ is trivial.
Since we are interested in the case when $\dK (G, G_1) \neq \emptyset$, thus we assume both $\mJ(G_1, G)$ and $\dK (G, G_1)$ are non-empty sets in the rest of the paper.

\begin{lemma}
\label{lem:psi_injective}
The map $\Psi$ in Definition \ref{def:psi}
is well-defined and injective.
\end{lemma}

\begin{proof}
First, we show the map $\Psi$ is well-defined. Consider any  point $(\bJ^*, \bx^*, \bp^*) \in \mJ(G_1,G)\times [(\bx_0 + \mS_{G_1} )\cap\mathbb{R}^n_{>0}] \times \mathbb{R}^b$.
From Proposition \ref{prop:craciun2020efficient}, if we set
$\bk_1 = (k_{1, \by\rightarrow \by'})_{\by\rightarrow \by' \in E_1}$
with
$J^*_{\by \rightarrow \by'} = k_{1, \by\rightarrow \by'} (\bx^*)^{\by}$, then 
\[
\bk_1 \in \mK_{\RR} (G_1,G) \subset \mK(G_1).
\]
Thus there exists $\bk \in \RR^{|E|}$, such that
$(G, \bk) \sim (G_1, \bk_1)$ and $\bk \in \dK(G,G_1)$.
Now we suppose $\bp^* = (p^*_1, p^*_2, \ldots, p^*_b)$ and set the vector $\bk^* \in \RR^{|E|}$ as 
\[
\bk^* = \bk + \sum\limits^{b}_{i=1} (p^*_i - \langle \bk, \bB_i \rangle ) \bB_i.
\]
Since $\{ \bB_i \}^b_{i=1}$ is an orthonormal basis of the subspace $\mD(G)$, we compute that
\begin{equation} \label{eq:k*p*}
\bk^* \bB_j = p^*_j,
\ \text{ for any }
1 \leq j \leq b.
\end{equation}
Moreover, from Lemma \ref{lem:d0} and
$\sum\limits^{b}_{i=1} (p^*_i - \bk \bB_i ) \bB_i \in \mD(G)$, we obtain
\[
(G, \bk^*) \sim (G_1, \bk_1).
\]

Now suppose there exists another $\bk^{**} \in \dK(G,G_1)$, such that
\[(G, \bk^{**}) \sim (G_1, \bk_1)
\ \text{ and } \
\bp^* = (\bk^{**} \bB_1, \bk^{**} \bB_2, \ldots, \bk^{**} \bB_b).
\]
Thus we have $(G, \bk^{**}) \sim (G, \bk^*)$, and this shows 
\[
\bk^{**} - \bk^{*} \in \mD(G).
\]
Together with \eqref{eq:k*p*}, we have
\[
p^*_j = \bk^* \bB_j = \bk^{**} \bB_j,
\ \text{ for any }
1 \leq j \leq b.
\]
Since $\{ \bB_i \}^b_{i=1}$ is an orthonormal basis of $\mD(G)$, we derive \[\bk^{**} = \bk^{*},\] 
and conclude that $\bk^* \in \dK(G,G_1)$ is well-defined.
Furthermore, from \eqref{def:psi_kq} we obtain 
\[
\bq^* = (\bJ^* \bA_1,\bJ^* \bA_2 ,..., \bJ^* \bA_a),
\] which is always well-defined.
Therefore, we have
\[
\Psi (\bJ^*, \bx^*, \bp^*) = (\bk^*, \bq^*)
\]
and $\Psi$ is well-defined.

\medskip

Second, we show the map $\Psi$ is injective. Suppose there exist two elements $(\bJ^*, \bx^*, \bp^*)$ and $(\bJ^{**}, \bx^{**}, \bp^{**})$ of $\mJ(G_1,G)\times [(\bx_0 + \mS_{G_1} )\cap\mathbb{R}^n_{>0}] \times \mathbb{R}^b$,
such that
\[
\Psi (\bJ^*, \bx^*, \bp^*) = \Psi (\bJ^{**}, \bx^{**}, \bp^{**}) = (\bk, \bq) \in \dK(G,G_1)\times Q.
\]
From $(\bJ^*, \bx^*)$ and $(\bJ^{**}, \bx^{**})$, we set
$\bk^* = (k^*_{\by\rightarrow \by'})_{\by\rightarrow \by' \in E_1}$ and $\bk^{**} = (k^{**}_{\by\rightarrow \by'})_{\by\rightarrow \by' \in E_1}$ as
\[
k^*_{\by\rightarrow \by'} = \frac{J^*_{\by\rightarrow \by'}}{{(\bx^*)}^{\by}},
\ \
k^{**}_{\by\rightarrow \by'} = \frac{J^{**}_{\by\rightarrow \by'}}{{(\bx^*)}^{\by}}.
\]
From Proposition \ref{prop:craciun2020efficient} and \eqref{def:psi_k}, we derive that 
\[\bk^*, \bk^{**} \in \mK_{\RR} (G_1,G)
\ \text{ and } \ 
(G, \bk) \sim (G_1, \bk^*) \sim (G_1, \bk^{**}).
\]
The uniqueness of the complex-balanced steady state within each invariant polyhedron implies
\[
\bx^* = \bx^{**}.
\]
Using Proposition \ref{prop:craciun2020efficient} and Lemma \ref{lem:j0}, we have 
\[
(G_1, \bJ^*) \sim (G_1, \bJ^{**})
\ \text{ and } \
\bJ^{**} - \bJ^* \in \eJ(G_1).
\]
Moreover, from \eqref{def:psi_kq} we obtain
\[
(\bk \bB_1, \bk \bB_2, \ldots, \bk \bB_b) = \bp^* =  \bp^{**},
\]
and
\[
\bq = (\bJ^* \bA_1, \bJ^* \bA_2, \ldots, \bJ^* \bA_a) 
= (\bJ^{**} \bA_1, \bJ^{**} \bA_2, \ldots, \bJ^{**} \bA_a).
\]
From $\{\bA_i \}^{a}_{i=1}$ is an orthonormal basis of the subspace $\eJ(G)$, together with $\bJ^{**} - \bJ^* \in \eJ(G_1)$, thus we get that 
\[
\bJ^* = \bJ^{**}.
\]
Therefore, we show $(\bJ^*, \bx^*, \bp^*) = (\bJ^{**}, \bx^{**}, \bp^{**})$ and conclude the injectivity.
\end{proof}

\begin{lemma}
\label{lem:psi_cts} 
The map $\Psi$ in Definition \ref{def:psi} is continuous.
\end{lemma}

\begin{proof}
Consider any fixed point $(\bJ, \bx, \bp) \in \mJ(G_1,G)\times [(\bx_0 + \mS_{G_1} )\cap\mathbb{R}^n_{>0}] \times \mathbb{R}^b$, such that
\[
\Psi(\bJ,\bx, \bp) = (\bk, \bq).
\]
From \eqref{def:psi_kq}, we have $\bq = (\bJ \bA_1,\bJ \bA_2 , \ldots, \bJ \bA_a)$ and $\bq$ is a continuous function of $\bJ$.

\smallskip

Now it suffices to show that $\bk$ varies continuously as a function of $(\bJ,\bx,\bq)$.
Note that $\bk$ is defined as
\[
(G, \bk) \sim (G_1, \bk_1) \ \text{ with } \ k_{1, \by\rightarrow \by'} = \frac{J_{\by\rightarrow \by'}}{{\bx}^{\by}}.
\]
From \eqref{eq:DE}, for every vertex $\by_0 \in V \cup V_1$
\begin{equation} \notag
\sum_{\by_0 \to \by \in E} k_{\by_0  \to \by} (\by - \by_0) 
= \sum_{\by_0 \to \by' \in E_1} k_{1, \by_0  \to \by'}  (\by' - \by_0).
\end{equation}
Together with \eqref{def:psi_kq}, we get for every vertex $\by_0 \in V \cup V_1$
\begin{equation} \label{eq:k_ct_1}
\sum_{\by_0 \to \by \in E} k_{\by_0  \to \by} (\by - \by_0) 
= \sum_{\by_0 \to \by' \in E_1} \frac{J_{\by_0 \rightarrow \by'}}{{\bx}^{\by_0}} (\by' - \by_0),
\end{equation}
and
\begin{equation} \label{eq:k_ct_2}
\bp = (\bk \bB_1, \bk \bB_2, \ldots, \bk \bB_b).
\end{equation}

Since $\bJ$, $\bx$, and $\bp$ are considered to be fixed, then we can rewrite \eqref{eq:k_ct_1} as
\begin{equation} \label{eq:k_ct_1_1}
\sum_{\by_0 \to \by \in E} k_{\by_0  \to \by} (\by - \by_0) 
= \text{constant}.
\end{equation}
Hence, the solutions to \eqref{eq:k_ct_1} form a linear subspace. 
Suppose $\bk'$ is another solution to \eqref{eq:k_ct_1}, then 
\[
(G, \bk) \sim (G, \bk').
\]
Using Lemma \ref{lem:d0}, we have $\bk' - \bk \in \mD (G)$. Together with the linearity, the tangent space to \eqref{eq:k_ct_1_1} at $(\bJ, \bx, \bp)$ is $\mD(G)$.
On the other hand, it is straightforward that the solutions to \eqref{eq:k_ct_2} form a linear subspace whose tangent space at $(\bJ, \bx, \bp)$ is tangential to 
\begin{equation} \notag
\spn \{\bB_1, \bB_2, \ldots, \bB_b\} = \mD(G).
\end{equation}

This shows two tangent spaces for~\eqref{eq:k_ct_1} and~\eqref{eq:k_ct_2} are complementary, and thus intersect transversally~\cite{guillemin2010differential}. From Lemma \ref{lem:psi_injective}, we get $\bk$ is the only solution to \eqref{eq:k_ct_1} and \eqref{eq:k_ct_2}. These indicate that the unique intersection point (solution) of two equations must vary continuously with respect to parameters $(\bJ, \bx, \bp)$. 
Therefore, $\bk$ varies continuously as a function of $(\bJ,\bx,\bq)$ and we conclude the lemma.
\end{proof}

\subsection{Lower Bound on the Dimension of the \texorpdfstring{$\RR$}{R}-Disguised Toric Locus}
\label{sec:lower_bound}

In this section, we estimate the dimension of the $\RR$-disguised toric locus. This will aid us to understand the size of the reaction rates of a given E-graph that admit a complex-balanced realization.

\medskip

Recall that a set $S$ is a \textbf{semialgebraic set} if it can be represented as a finite union of sets defined by polynomial equalities and polynomial inequalities. On a dense\footnote{As  usual, a subset $B$ of a topological space $X$ is {\em dense} if the closure of $B$ is equal to $X$.} open subset of the semialgebraic set S, it is locally a \textbf{submanifold}~\cite{lee2010introduction}. One can define the \textbf{dimension} of S to be the largest dimension at points at which it is a submanifold. Moreover, semialgebraic sets are closed under finite unions or intersections, the projection operation, and the polynomial mapping~\cite{bierstone1988semianalytic}.

\begin{lemma}
\label{lem:semi_algebaic}
Consider a weakly reversible E-graph $G_1 = (V_1, E_1)$. For any E-graph $G = (V, E)$, $\dK(G,G_1)$ is a semialgebraic set.
\end{lemma} 

\begin{proof}
Note that if $\dK(G, G_1) = \emptyset$, then we immediately conclude the lemma. Now suppose $\dK(G, G_1) \neq \emptyset$.
From Definition \ref{def:de_realizable}, for any $\bk \in \dK (G, G_1)$, there exists $\bk_1 \in \mK (G_1, G)$ such that
\[
(G, \bk) \sim (G_1, \bk_1).
\]
Thus for every vertex $\by_0 \in V_1 \ / \ V$
\begin{equation} \label{eq:semi_k_g1g_1}
\sum_{\by_0 \to \by' \in E_1} k_{1, \by_0  \to \by'}  (\by' - \by_0) = \mathbf{0},
\end{equation}
and for every vertex $\by_0 \in V \cap V_1$
\begin{equation} \label{eq:semi_k_g1g_2}
 \sum_{\by_0 \to \by' \in E_1} k_{1, \by_0  \to \by'}  (\by' - \by_0)
= \sum_{\by_0 \to \by \in E} k_{\by_0  \to \by} (\by - \by_0).
\end{equation}

Recall from Theorem \ref{thm:homeo} that $\mK (G_1)$ is a toric variety, and thus it must be a semialgebraic set.
Note that $\mK (G_1, G) \subseteq \mK (G_1)$ in which each reaction rate vector $\bk_1\in \mK (G_1, G)$ satisfies Equations ~\eqref{eq:semi_k_g1g_1} and~\eqref{eq:semi_k_g1g_2}. 
From~\eqref{eq:semi_k_g1g_1}, $\bk_1$ is defined by polynomial equalities on every vertex $\by_0 \in V_1 \ / \ V$. Then since $\bk$ can be picked arbitrarily in $\RR^E$ in \eqref{eq:semi_k_g1g_2}, Equation~\eqref{eq:semi_k_g1g_2} is equivalent to for every vertex $\by_0 \in V \cap V_1$
\begin{equation} \notag
\sum_{\by_0 \to \by' \in E_1} k_{1, \by_0  \to \by'}  (\by' - \by_0) \in \spn \{ \by - \by_0 \}_{\by_0 \to \by \in E}.
\end{equation}
Thus, $\bk_1$ is also defined by polynomial equalities from \eqref{eq:semi_k_g1g_2} on every vertex $\by_0 \in V$. 
Therefore, we show that $\mK (G_1, G)$ is a semialgebraic set.

Now we express $\bk \in \dK (G, G_1)$ in terms of $\bk_1 \in \mK (G_1, G)$ from \eqref{eq:semi_k_g1g_2}. Without loss of generality we assume that there exists $r$ edges in $G = (V, E)$ from the vertex $\by_0 \in V$ as follows:
\[
\{ \by_0 \to \by_1, \by_0 \to \by_2, \ldots, \by_0 \to \by_r \} \subset E.
\]
Further, we assume the first $s$ reaction vectors $\{ \by_i - \by_0 \}^{s}_{i=1}$ form a basis of $\spn \{ \by_j - \by_0 \}^r_{j=1}$, and there exist $r-s$ vectors
$\bv_1, \bv_2, \ldots, \bv_{r-s}$, such that
\begin{equation}
\ker (\bY) = \spn \{ \bv_1, \bv_2, \ldots, \bv_{r-s} \},
\end{equation}
where $\bY$ is a $n \times r$ matrix whose $i^{th}$ column is the reaction vector $\by_i - \by_0$.

Let $\tilde{\bY}$ be the matrix containing the first $s$ columns of $\bY$, then $\big( \tilde{\bY} \big)^T \tilde{\bY}$ is invertible  and the solution set of reaction rate vectors $\bk \in \dK (G, G_1)$ to \eqref{eq:semi_k_g1g_2} can be written as 
\begin{equation} \label{eq:semi_k_g1g_3}
\Big( \big( \tilde{\bY} \big)^T \tilde{\bY} \Big)^{-1}
\big( \tilde{\bY} \big)^T \big( \sum_{\by_0 \to \by' \in E_1} k_{1, \by_0  \to \by'}  (\by' - \by_0) \big)
+ \sum\limits^{r-s}_{i=1} a_i \bv_i,
\end{equation}
with $a_i \in \RR$ for $1 \leq i \leq r-s$. 
Hence, $\dK (G, G_1)$ can be defined by a polynomial mapping from $\mK (G_1, G) \times \RR^{r-s}$.
Since we show $\mK (G_1, G)$ is a semialgebraic set in $\RR^{|E_1|}$, it is clear that $\mK (G_1, G) \times \RR^{r-s}$ is a semialgebraic set in $\RR^{|E_1| + (r-s)}$.
Therefore we conclude that $\dK (G, G_1)$ is semialgebraic set in $\RR^{|E|}$.
\end{proof}


Now we are ready to prove our main theorem.

\begin{theorem}
\label{thm:dim_kisg}

Let $G_1 = (V_1, E_1)$ be a weakly reversible E-graph with its stoichiometric subspace $\mS_{G_1}$. Consider an E-graph $G = (V, E)$ and $\bx_0\in\mathbb{R}^n_{>0}$, then
\begin{equation} \label{eq:dim_kisg}
\dim(\dK(G,G_1)) \geq \dim (\mJ(G_1,G)) + \dim (\mS_{G_1})  + \dim(\eJ(G_1)) - \dim(\mD(G)),
\end{equation}
where $\dK(G,G_1)$ and $\mJ(G_1,G)$ are defined in Definitions \ref{def:de_realizable} and \ref{def:flux_realizable}; $\mD(G)$ and $\eJ(G_1)$ are defined in Definitions~\ref{def:d0} and \ref{def:j0} respectively.
\end{theorem}


\begin{proof}

As before, we let $b = \dim(\mD(G))$ and let $\{\bB_1, \bB_2, \ldots, \bB_b\}$ be an orthonormal basis of $\mD(G)$.
We also let $a = \dim(\eJ(G_1))$ and let $\{\bA_1, \bA_2, \ldots, \bA_a \}$ be an orthonormal basis of $\eJ(G_1)$.

\medskip

\textbf{Step 1: }
First, we define the set $\hat{\mJ} (G_1,G) \subset \RR^{|E_1|}$ as 
\begin{equation}
\label{def:hat_j_g1_g}
\hat{\mJ} (G_1,G) = \{ \bJ + \sum\limits^a_{i=1} w_i \bA_i \ | \ \bJ \in \mJ (G_1,G), \text{ and } w_i \in \RR \text{ for } 1 \leq i \leq a \}.
\end{equation}
It is clear that $\mJ (G_1,G) \subset \hat{\mJ} (G_1,G)$, so we introduce an 
extension of $\Psi$ from \eqref{eq:psi} as follows:
\begin{equation} \label{eq:hpsi}
\hat{\Psi}: \hat{\mJ} (G_1,G) \times [(\bx_0 + \mS_{G_1} )\cap\mathbb{R}^n_{>0}] \times \RR^b \rightarrow \dK(G,G_1)\times \RR^a,
\end{equation}
such that for $(\hat{\bJ}, \bx, \bp) \in \hat{\mJ} (G_1,G) \times [(\bx_0 + \mS_{G_1} )\cap\mathbb{R}^n_{>0}] \times \mathbb{R}^b$, 
\begin{equation} \notag
\hat{\Psi} (\hat{\bJ},\bx, \bp) 
: = (\hat{\bk}, \hat{\bq}),
\end{equation}
where
\begin{equation} \label{def:hpsi_k}
(G, \hat{\bk}) \sim (G_1, \hat{\bk}_1) \ \text{ with } \ \hat{k}_{1, \by\rightarrow \by'} = \frac{\hat{J}_{\by\rightarrow \by'}}{{\bx}^{\by}}.
\end{equation} 
and
\begin{equation} \label{def:hpsi_kq}
\bp = (\hat{\bk} \bB_1, \hat{\bk} \bB_2, \ldots, \hat{\bk} \bB_b), 
\ \
\hat{\bq} = (\hat{\bJ} \bA_1, \hat{\bJ} \bA_2, \ldots, \hat{\bJ} \bA_a).
\end{equation}

Note that the map $\hat{\Psi}$ defined in \eqref{eq:hpsi}-\eqref{def:hpsi_kq} is analogous to $\Psi$ in Definition \ref{def:psi}.
Thus using similar arguments in Lemma \ref{lem:psi_injective} and Lemma \ref{lem:psi_cts}, we can deduce that the map $\hat{\Psi}$ is well-defined, injective, and continuous. Now we show the map $\hat{\Psi}$ is also surjective. 

Consider any fixed point $(\hbk, \hbq) \in \dK(G,G_1)\times \RR^a$.
Since $\hbk \in \dK (G, G_1)$, there exists $\bk_1 \in \mK (G_1, G)$ such that
\begin{equation} \label{eq:gk_g1k1}
(G, \bk) \sim (G_1, \bk_1) 
\ \text{ with } \ \bk_1 = (k_{1, \by\rightarrow \by'})_{\by\rightarrow \by' \in E_1}.
\end{equation}
Theorem \ref{thm:cb} shows that the complex-balanced system $(G_1, \bk_1)$ has a unique steady state $\bx \in [(\bx_0 + \mS_{G_1} )\cap\mathbb{R}^n_{>0}]$. 
We set the flux vector $\bJ_1$ as 
\[
\bJ_1 = (J_{1, \by\rightarrow \by'})_{\by\rightarrow \by' \in E_1}
\ \text{ with } \ J_{1, \by\rightarrow \by'} = k_{1, \by\rightarrow \by'} {\bx}^{\by}.
\]
By \eqref{def:hpsi_k}, the flux system $(G, \bJ_1)$ gives rise to $(G_1, \bk_1)$.
Now suppose $\hbq = (\hat{q}_1, \hat{q}_2, \ldots, \hat{q}_a)$, we build a new flux vector $\hbJ$ as follows
\[
\hbJ = \bJ_1 + \sum\limits^{a}_{i=1} (\hat{q}_i - \langle \bJ_1, \bA_i \rangle ) \bA_i.
\]
Since $\{ \bA_i \}^a_{i=1}$ is an orthonormal basis of the subspace $\eJ(G_1)$, we get for $1 \leq i \leq a$
\begin{equation} \notag
\hbJ \bA_i = \hat{q}_i.
\end{equation}
Using Lemma \ref{lem:j0} and $\sum\limits^{a}_{i=1} (\hat{q}_i - \langle\bJ_1 \bA_i\rangle ) \bA_i \in \eJ(G_1)$, we obtain
$(G, \hbJ) \sim (G_1, \bJ_1)$. Let $\hbk_1 = (k_{1, \by\rightarrow \by'})_{\by\rightarrow \by' \in E_1}$ with $\hat{k}_{1, \by\rightarrow \by'} = \frac{\hat{J}_{\by\rightarrow \by'}}{{\bx}^{\by}}$, from Proposition \ref{prop:craciun2020efficient} and \eqref{eq:gk_g1k1} we have
\[
(G, \hbk) \sim (G_1, \bk_1) \sim (G, \hbk_1).
\]
Further, we let $\bp = (\hat{\bk} \bB_1, \hat{\bk} \bB_2, \ldots, \hat{\bk} \bB_b)$ and derive
\[
\hat{\Psi} (\hat{\bJ},\bx, \bp) = (\hat{\bk}, \hat{\bq}).
\]
Therefore, we prove the map $\hat{\Psi}$ is surjective. 

\medskip

\textbf{Step 2: }
From Lemma \ref{lem:j_g1_g_cone}, there exists a set of vectors $\{ \bv_1, \bv_2, \ldots, \bv_k \}$, such that
\[
\mJ (G_1, G) = \{ a_1 \bv_1 + \cdots a_k \bv_k \ | \ a_i \in \RR_{>0}, \bv_i \in \RR^n\}.
\]
Moreover, $\eJ(G_1) \subseteq \spn \{ \bv_1, \bv_2, \ldots, \bv_k \}$ and
$\dim (\mJ (G_1, G)) =\dim ( \spn \{ \bv_1, \bv_2, \ldots, \bv_k \} )$.

Following \eqref{def:hat_j_g1_g}, we derive that $\hat{\mJ} (G_1, G)$ can be represented as the positive combination of the following vectors: 
\begin{equation} \label{hj_g1g_basis}
\{ \bv_1, \bv_2, \ldots, \bv_k, \pm \bA_1, \pm \bA_2, \ldots, \pm \bA_a \}. 
\end{equation}
Pick any vector $\hat{\bJ}_1 \in \hat{\mJ} (G_1, G) \subset \RR^n$, 
from \eqref{def:hat_j_g1_g} there exist vector $\bJ_1 \in \mJ (G_1, G)$ and $\br_1 \in \spn \{ \bA_i \}^a_{i=1} = \eJ(G_1)$, such that
\[
\hat{\bJ}_1 = \bJ_1 + \br_1.
\]
From \eqref{j_g1_g_generator_2} in Lemma \ref{lem:j_g1_g_cone}, there must exist sufficiently small number $\varepsilon > 0$, such that
\begin{equation} \notag
\bJ_1 + b_1 \bv_1 + \cdots + b_k \bv_k \in \mJ (G_1, G), 
\ \text{ for any }
|b_i| \leq \varepsilon
 \text{ with } 
1 \leq i \leq k.
\end{equation}
Using \eqref{def:hat_j_g1_g}, we get that for any $\br \in \spn \{ \bA_i \}^a_{i=1}$,
\begin{equation} 
\label{hat_j_g1_g_neighbour}
\hat{\bJ}_1 + \sum\limits^k_{i=1} b_i \bv_i + \br 
= \big( \bJ_1 + \sum\limits^k_{i=1} b_i \bv_i \big) + ( \br_1 + \br)
\in \mJ (G_1, G),
\end{equation}
and every neighbourhood of $\hat{\bJ}_1$ is in the form of \eqref{hat_j_g1_g_neighbour}.
Hence, we derive that
\begin{equation} \notag
\dim (\hat{\mJ} (G_1, G)) =\dim ( \spn \{ \bv_1, \bv_2, \ldots, \bv_k, \bA_1, \bA_2, \ldots, \bA_a \} ).
\end{equation}
Since $\spn \{ \bA_i \}^a_{i=1} = \eJ(G_1) \subseteq \spn \{ \bv_1, \bv_2, \ldots, \bv_k \}$, we have
\begin{equation} \notag
\dim (\hat{\mJ} (G_1, G)) = \dim ( \spn \{ \bv_1, \bv_2, \ldots, \bv_k \} ).
\end{equation}
Further, using $\dim (\mJ (G_1, G)) = \dim ( \spn \{ \bv_1, \bv_2, \ldots, \bv_k \} )$ in Lemma \ref{lem:j_g1_g_cone}, we derive 
\begin{equation} \label{hat_j_g1_g_generator_dim}
\dim (\hat{\mJ} (G_1, G)) = \dim (\mJ (G_1, G)).
\end{equation}

\medskip

\textbf{Step 3: }
From Lemma \ref{lem:semi_algebaic}, we get $\dK(G,G_1)$ is a semialgebraic set. Then $\dK(G, G_1)$ is locally a submanifold on a dense open subset. 
Recall that the dimension of a semialgebraic set is the largest dimension at points of which it is locally a submanifold.
Hence there exists some $\bk \in \dK(G, G_1)$ and a neighbourhood of $\bk$, denoted by $U$, such that 
\[
\bk \in U \subset \dK(G, G_1),
\]
where $U$ is a submanifold with $\dim (U) = \dim (\dK(G, G_1))$. 
Then we pick an open set $V$ of $\RR^a$ and let $B = U \times V$.
It is clear that $B$ is open in $\dK(G,G_1)\times \RR^a$ and
\begin{equation}
\label{eq:B_dim}
\dim (B) = \dim (\dK(G, G_1)) + a.
\end{equation}

Since we show that $\hat{\Psi}$ defined in \eqref{eq:hpsi} is surjective and injective, then we consider the preimage of $B$ under the map $\hat{\Psi}$, denoted by $A = \hat{\Psi}^{-1} (B)$. Since $\hat{\Psi}$ is also continuous, $A$ is a open set in $\hat{\mJ} (G_1,G) \times [(\bx_0 + \mS_{G_1} )\cap\mathbb{R}^n_{>0}] \times \RR^b$. 
Recall that $\hat{\mJ} (G_1, G)$ satisfies \eqref{hj_g1g_basis},
$(\bx_0 + \mS_{G_1} )\cap\mathbb{R}^n_{>0}$ is the intersection of an affine linear subspace and the positive orthant, and $\RR^b$ is a $b$-dimensional Euclidean space, we derive 
\begin{equation} \label{eq:A_dim}
\dim (A) = \dim (\hat{\mJ} (G_1, G)) + \dim (\mS_{G_1}) + b.
\end{equation}

Using the fact that $\hat{\Psi}$ is injective and continuous, the invariance of dimension theorem~\cite{hatcher2005algebraic,munkres2018elements} shows that $\dim (A) \leq \dim (B)$.
Together with \eqref{hat_j_g1_g_generator_dim}, \eqref{eq:B_dim}, and \eqref{eq:A_dim}, we obtain
\begin{equation}
\dim (\mJ (G_1, G)) + \dim (\mS_{G_1}) + b
\leq \dim (\dK(G, G_1)) + a.
\end{equation}
and conclude \eqref{eq:dim_kisg}.
\end{proof}

\begin{theorem}
\label{thm:dim_kisg 2}
Let $G = (V, E)$ be an E-graph. 
Suppose $G_1 = (V_1, E_1)$ is a weakly reversible E-graph
with its stoichiometric subspace $\mS_{G_1}$ and $\dK(G,G_1) \neq \emptyset$. Then 
\[
\dim (\dK(G)) \geq \dim (\mJ(G_1,G)) + \dim (\mS_{G_1}) + \dim(\eJ(G_1)) - \dim(\mD(G)).
\]
\end{theorem}

\begin{proof}
Since $\dK(G, G_1) \subset \dK(G)$, the proof follows directly from Theorem \ref{thm:dim_kisg}.
\end{proof}

\begin{example}
We consider the $\RR$-disguised toric locus on some pairs of E-graphs in Figure~\ref{fig:graphs_dim_dis}.
Specifically, we show how Theorem~\ref{thm:dim_kisg} leads to the lower bound of the $\RR$-disguised toric locus.

\begin{figure}[H]
\centering
\includegraphics[scale=0.34]{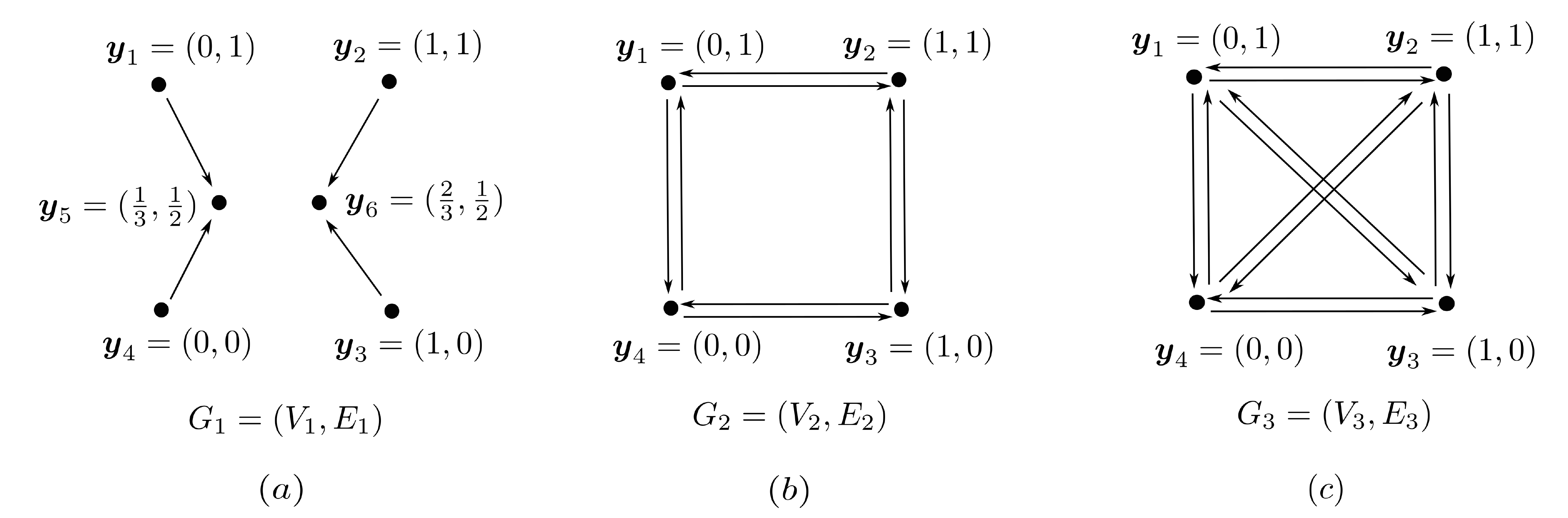}
\caption{Three E-graphs $G_1=(V_1,E_1), G_2 = (V_2,E_2)$ and $G_3=(V_3,E_3)$. Note that $G_1 \subseteq G_3$ and $G_2 \sqsubseteq G_3$.}
\label{fig:graphs_dim_dis}
\end{figure} 

Note that every flux vector in $\mJ(G, G')$ needs to be a complex-balanced vector in $G$, where the total incoming flux and total outgoing flux on each vertex are equal.
\cite{craciun2020structure} shows such flux-balanced conditions give ``the number of vertices minus one'' linear independent constraints on the flux vector.
Further, flux vectors in $\mJ(G, G')$ also need to be $\RR$-realizable in $G'$, which depends on the graphic relation between $G$ and $G'$. 

\begin{enumerate}
\item[(a)] \textbf{For $\dK(G_1, G_2)$.}
We can check that $G_2$ has four source vertices and eight reactions with a two-dimensional stoichiometric subspace. 

We start with computing $\dim (\mJ(G_2, G_1))$.
Consider any flux vector $\bJ \in \mJ(G_2, G_1) \subset \RR^8$.
First, being a complex-balanced vector in $G_2$ gives $4-1=3$ constraints on $\bJ$. 
Second, being $\RR$-realizable in $G_1$ gives $4$ constraints (one constraint for each vertex) on $\bJ$.
Thus we obtain 
\[
\dim (\mJ(G_2, G_1)) \geq 8 - 4 - 3 = 1.
\]
From Example \ref{ex:D0J0G2G3} and checking graph $G_1$, we get 
\[
\mD (G_1) = \eJ (G_2) = \{ \mathbf{0} \}.
\]
Therefore,  by Theorem \ref{thm:dim_kisg} we conclude that
\begin{equation}
\begin{split} \notag
\rm{dim}(\dK(G_1, G_2)) 
& \geq \dim (\mJ(G_2, G_1)) + \dim (\mathcal{S}_{G_2}) + \dim(\mD (G_1)) - \dim(\eJ (G_2)) 
\\& \geq 1 + 2 + 0 - 0 = 3.
\end{split}
\end{equation}

\item[(b)] \textbf{For $\dK(G_1, G_3)$.}
We can check that $G_3$ has four source vertices and twelve reactions, with a two-dimensional stoichiometric subspace. 

We start with computing $\dim (\mJ(G_3, G_1))$.
Consider any flux vector $\bJ \in \mJ(G_3, G_1) \subset \RR^{12}$.
First, being a complex-balanced vector in $G_3$ gives $4-1=3$ constraints on $\bJ$. 
Second, being $\RR$-realizable in $G_1$ gives $4$ constraints (one constraint for each vertex) on $\bJ$.
Thus we obtain 
\[
\dim (\mJ(G_3, G_1)) \geq 12 - 4 - 3 = 5.
\]
From Example \ref{ex:D0J0G2G3} and checking graph $G_1$, we get 
\[
\mD (G_1) = \{ \mathbf{0} \}
\ \text{ and } \
\dim (\eJ(G_3)) =  3.
\]
By Theorem \ref{thm:dim_kisg}, we derive that
\begin{equation}
\begin{split} \notag
\rm{dim}(\dK(G_1, G_3)) 
& \geq \dim (\mJ(G_3, G_1)) + \dim (\mathcal{S}_{G_3}) + \dim(\mD (G_1)) - \dim(\eJ (G_3)) 
\\& \geq 5 + 2 + 0 - 3 = 4.
\end{split}
\end{equation}
Note that $\dK(G_1, G_3) \subset \RR^4$, therefore we can conclude
\[
\dim (\mJ(G_3, G_1)) = 5
\ \text{ and } \
\rm{dim}(\dK(G_1, G_3)) = 4.
\]
Therefore we conclude that $\rm{dim}(\dK(G_1)) = 4$.

\item[(c)] \textbf{For $\dK(G_2, G_3)$.}
We start with computing $\dim (\mJ(G_3, G_2))$.
Consider any flux vector $\bJ \in \mJ(G_3, G_2) \subset \RR^{12}$.
First, being a complex-balanced vector in $G_3$ gives $4-1=3$ constraints on $\bJ$. 
Second, being $\RR$-realizable in $G_2$ gives no constraints on $\bJ$ since every flux vector can be transformed into $G_2$.
Thus we obtain 
\[
\dim (\mJ(G_3, G_2)) \geq 12 - 3 - 0 = 9.
\]
From Example \ref{ex:D0J0G2G3}, we get 
\[
\mD (G_2) = \{ \mathbf{0} \}
\ \text{ and } \
\dim (\eJ(G_3)) =  3.
\]
By Theorem \ref{thm:dim_kisg}, we derive that
\begin{equation}
\begin{split} \notag
\rm{dim}(\dK(G_2, G_3)) 
& \geq \dim (\mJ(G_3, G_2)) + \dim (\mathcal{S}_{G_3}) + \dim(\mD (G_2)) - \dim(\eJ (G_3)) 
\\& \geq 9 + 2 + 0 - 3 = 8.
\end{split}
\end{equation}
Note that $\dK(G_1, G_3) \subset \RR^8$, therefore we can conclude
\[
\dim (\mJ(G_3, G_1)) = 9
\ \text{ and } \
\rm{dim}(\dK(G_1, G_3)) = 8.
\]
Therefore we conclude that $\rm{dim}(\dK(G_2)) = 8$.

\item[(d)] \textbf{For $\dK(G_3, G_3)$.}
We start with computing $\dim (\mJ(G_3, G_3))$.
Consider any flux vector $\bJ \in \mJ(G_3, G_2) \subset \RR^{12}$.
Being a complex-balanced vector in $G_3$ gives $4-1=3$ constraints on $\bJ$ and clearly every flux vector is $\RR$-realizable in $G_3$.
Thus we obtain 
\[
\dim (\mJ(G_3, G_2)) \geq 12 - 3 - 0 = 9.
\]
From Example \ref{ex:D0J0G2G3}, we get 
\[
\dim (\mD (G_3)) = 4
\ \text{ and } \
\dim (\eJ(G_3)) =  3.
\]
By Theorem \ref{thm:dim_kisg}, we derive that
\begin{equation}
\begin{split} \notag
\rm{dim}(\dK(G_3, G_3)) 
& \geq \dim (\mJ(G_3, G_3)) + \dim (\mathcal{S}_{G_3}) + \dim(\mD (G_3)) - \dim(\eJ (G_3)) 
\\& \geq 9 + 2 + 4 - 3 = 12.
\end{split}
\end{equation}
Note that $\dK(G_3, G_3) \subset \RR^{12}$, therefore we can conclude
\[
\dim (\mJ(G_3, G_1)) = 9
\ \text{ and } \
\rm{dim}(\dK(G_1, G_3)) = 12.
\]

Therefore we conclude that $\rm{dim}(\dK(G_3)) = 12$.

\end{enumerate}
\end{example}

\begin{remark}
    In the example above we have chosen, for simplicity, some networks with especially simple geometry (e.g., the source complexes of $G_1$ are the vertices of a square). On the other hand,  the same approach would have worked without change as long as (for example) the source complexes of $G_1$ were the vertices of {\em convex nondegenerate quadrilateral}, and its four reactions pointed towards {\em the interior} of that quadrilateral. In particular, the precise position of the {\em target} vertices  of $G_1$ is irrelevant, as long as they lie in the interior of that quadrilateral.
\end{remark}

\begin{remark}
Note that, given a weakly reversible graph $G$, in general we have that $\dK(G, G)$ is {\em not} the same as its toric locus $\mK(G)$. This may look surprising, because the meaning of $\dK(G, G)$ is ``parameter values on $G$ that can be realized as toric using the graph $G$"; recall though that the same graph may give rise to {\em multiple realizations} of a dynamical system, and this is why $\dK(G, G)$ can be larger than $\mK(G)$. Indeed, we have already seen such an example above: the case of the complete graph $G_3$, for which $\mK(G_3)$ is a variety of codimension one, while $\dK(G_3, G_3)$ is actually the {\em $\RR$-disguised} toric locus of $G_3$, and is a full-dimensional subset of $\RR^{12}$. 
\end{remark}

\section{Discussion}
\label{sec:discussion}

Complex-balanced systems (also known as toric dynamical systems) exhibit a wide range of robust dynamical properties. In particular, their positive steady states are locally asymptotically stable. Further, it is known that for these systems, there exists a unique positive steady in each stoichiometric compatibility class~\cite{horn1972general},  and it is conjectured to be a global attractor~\cite{craciun2015toric}. Here we focus on \emph{$\RR$-disguised toric dynamical systems} - i.e., dynamical systems that have the same polynomial right-hand side as some toric dynamical systems. 

In Theorem~\ref{thm:dim_kisg} we provide a 
lower bound on the dimension of the $\RR$-disguised toric locus. In several examples, we show that reaction networks for which the toric locus has measure zero can have an $\RR$-disguised toric locus that has a positive measure. This allows us to address a {\em significant limitation of the classical theory of complex balanced systems}: the fact that, for networks of positive deficiency, the set of parameters for which we can take advantage of this theory has positive codimension, so it is, in a sense, small. We see here that, in some cases,  {\em this limitation can be overcome if we extend our analysis to the $\RR$-disguised toric systems}.

The approach described here lays the foundation for some future projects. In upcoming work~\cite{disg_2}, we investigate under what conditions the lower bound on the dimension of the $\RR$-disguised toric locus obtained in this paper is actually the {\em exact} value of this dimension. Another potential research direction is to explore further the properties of the continuous and injective map $\Psi$ defined in Equation~\ref{eq:psi}. For example, if  $\Psi$ (or a well-chosen restriction of $\Psi$) can be shown to be a diffeomorphism, this may allow us to identify subsets of the $\RR$-disguised toric locus that are smooth manifolds and realize its dimension. This will help us understand better under what conditions on the network  we can guarantee that its $\RR$-disguised toric locus has a positive measure. Yet another direction we plan to explore in upcoming work~\cite{disg_3} is to related the dimension of the $\RR$-disguised toric locus to the dimension of the (usual) disguised toric locus.

\section*{Acknowledgements} 

This work was supported in part by the National Science Foundation grant DMS-2051568.

\bibliographystyle{unsrt}
\bibliography{Bibliography}

\end{document}